\documentclass[12pt]{amsart}
\usepackage{amssymb,amscd}
\usepackage{verbatim}

\usepackage{amsmath,amssymb,graphicx,mathrsfs}   
\usepackage{enumerate}
\usepackage[colorlinks=true,allcolors = blue]{hyperref} 

\usepackage{tikz}
\usetikzlibrary{matrix}

\usepackage[all]{xy}

\let\frak\mathfrak

\def\>{\relax\ifmmode\mskip.666667\thinmuskip\relax\else\kern.111111em\fi}
\def\<{\relax\ifmmode\mskip-.333333\thinmuskip\relax\else\kern-.0555556em\fi}
\def\vsk#1>{\vskip#1\baselineskip}
\def\vv#1>{\vadjust{\vsk#1>}\ignorespaces}
\def\vvn#1>{\vadjust{\nobreak\vsk#1>\nobreak}\ignorespaces}
 
\def\fratop{\genfrac{}{}{0pt}1}
\def\satop#1#2{\fratop{\scriptstyle#1}{\scriptstyle#2}}
  \let\ssize\scriptstyle
\let\sssize\scriptscriptstyle

\let\Medskip\medskip
\def\medskip{\par\Medskip}
\let\Bigskip\bigskip
\def\bigskip{\par\Bigskip}

\let\Maketitle\maketitle
\def\maketitle{\Maketitle\thispagestyle{empty}\let\maketitle\empty}

\newtheorem{thm}{Theorem}[section]
\newtheorem{cor}[thm]{Corollary}
\newtheorem{lem}[thm]{Lemma}

\newtheorem{conj}[thm]{Conjecture}
\newtheorem{defn}[thm]{Definition}

\theoremstyle{definition}                                  
\numberwithin{equation}{section}

\theoremstyle{definition}
\newtheorem*{rem}{Remark}

\let\mc\mathcal
\let\nc\newcommand

\let\al\alpha

\let\ka\kappa
\let\la\lambda
\let\La\Lambda

\let\phi\varphi
\let\si\sigma

\let\om\omega
\let\Om\Omega

\let\der\partial

\let\ox\otimes

\let\ge\geqslant
\let\geq\geqslant
\let\le\leqslant
\let\leq\leqslant

\let\on\operatorname
\let\bi\bibitem
\let\bs\boldsymbol

\def\C{{\mathbb C}}
\def\Z{{\mathbb Z}}
\def\R{{\mathbb R}}

\def\F{{\mathbb F}}   % new Dec 2019

\def\+#1{^{\{#1\}}}

\def\beq{\begin{equation}}
\def\eeq{\end{equation}}
\def\be{\begin{equation*}}
\def\ee{\end{equation*}}

\nc{\bea}{\begin{eqnarray*}}
\nc{\eea}{\end{eqnarray*}}
\nc{\bean}{\begin{eqnarray}}
\nc{\eean}{\end{eqnarray}}
\def\g{{\mathfrak g}}

\let\ga\gamma

\nc{\Il}{{\mc I_{\bs\la}}}
\nc{\bla}{{\bs\la}}
\nc{\Fla}{\F_\bla}
\nc{\tfl}{{T^*\Fla}}
\nc{\GL}{{GL_n(\C)}}
\nc{\GLC}{{GL_n(\C)\times\C^*}}

\let\sd s %% \def\sd{\dot s}

\def\ddk_#1{\kk_{#1}\<\>\frac\der{\der\<\>\kk_{#1}}}

\def\bul{\mathbin{\raise.2ex\hbox{$\sssize\bullet$}}}
\def\intt{\mathchoice
{\mathop{\raise.2ex\rlap{$\,\,\ssize\backslash$}{\intop}}\nolimits}
{\mathop{\raise.3ex\rlap{$\,\sssize\backslash$}{\intop}}\nolimits}
{\mathop{\raise.1ex\rlap{$\sssize\>\backslash$}{\intop}}\nolimits}
{\mathop{\rlap{$\sssize\<\>\backslash$}{\intop}}\nolimits}}

\let\kk q %% Q
\let\cc c

\let\Ko K

\def\GZ/{Gelfand-Zetlin}
\def\KZ/{{\slshape KZ\/}}
\def\qKZ/{{\slshape qKZ\/}}
\def\XXX/{{\slshape XXX\/}}

\def\Sym{\on{Sym}}

\nc{\A}{{\mc C}}

\def\sll{{\frak{sl}}}

\def\Q{{\mathbb Q}}

\nc{\hsl}{\widehat{{\frak{sl}_2}}}

\nc{\BC}{{ \mathbb C}}
\nc{\lra}{\longrightarrow}
\nc{\CO}{{\mathcal{O}}}
\nc{\BZ}{{ \mathbb Z}}
\nc{\hfn}{\hat{\frak{n}}}
\nc\Zs{{\Z/p^s\Z}}
\nc\Zo{{\Zs[z]^0}}
\nc\gr{{\on{gr}}}

\nc\fD{{\frak D}}

\begin{document}

\hrule width0pt
\vsk->

\title[Ghosts and congruences]
{Ghosts and congruences for $p^s$-approximations 
\\
of hypergeometric periods}

\author[A.\:Varchenko and W.\,Zudilin]
{Alexander Varchenko$^{\star}$ and Wadim Zudilin$^{\diamond}$}

\maketitle

\begin{center}
{\it $^{\star}$ Department of Mathematics, University
of North Carolina at Chapel Hill\\ Chapel Hill, NC 27599-3250, USA\/}

\vsk.5>
{\it $^{ \star}$ Faculty of Mathematics and Mechanics, Lomonosov Moscow State
University\\ Leninskiye Gory 1, 119991 Moscow GSP-1, Russia\/}

\vsk.5>
 {\it $^{ \star}$ Moscow Center of Fundamental and Applied Mathematics
\\ Leninskiye Gory 1, 119991 Moscow GSP-1, Russia\/}

\vsk.5>
 {\it $^{ \diamond}$ 
Institute for Mathematics, Astrophysics and Particle Physics, Radboud University,
\\
 PO Box 9010, 6500 GL, Nijmegen, The Netherlands\/}

\end{center}

\vsk>
{\leftskip3pc \rightskip\leftskip \parindent0pt \Small
{\it Key words\/}:  Hypergeometric equation; KZ equations; Dwork congruences; master polynomials; $p^s$-approximation polynomials.

\vsk.6>
{\it 2020 Mathematics Subject Classification\/}: 11D79 (12H25, 32G34, 33C05, 33E30)
\par}
% 11A07 Congruences; primitive roots; residue systems
% 11B65 Binomial coefficients; factorials; q-identities
% 11D79 Congruences in many variables
% 12H25 p-adic differential equations
% 32G34 Moduli and deformations for ordinary differential equations (e.g., Knizhnik-Zamolodchikov equation)
% 33C05 Classical hypergeometric functions, 2F1
% 33C60 Hypergeometric integrals and functions defined by them (E, G, H and I functions)
% 33C70 Other hypergeometric functions and integrals in several variables
% 33C75 Elliptic integrals as hypergeometric functions
% 33E05 Elliptic functions and integrals
% 33E30 Other functions coming from differential, difference and integral equations
% 33E50 Special functions in characteristic p (gamma functions, etc.)

{\let\thefootnote\relax
\footnotetext{\vsk-.8>\noindent
$^\star\<${\sl E\>-mail}:\enspace anv@email.unc.edu,
supported in part by NSF grant DMS-1954266
\\
$^\diamond\<${\sl E\>-mail}:\enspace  w.zudilin@math.ru.nl
}}

\begin{abstract}

We prove general Dwork-type congruences for constant terms attached to tuples of Laurent polynomials.
We apply this result to establishing arithmetic and $p$-adic analytic properties of functions
 originating from polynomial solutions modulo $p^s$ of hypergeometric and KZ equations, 
 solutions which come as coefficients of master 
polynomials and whose coefficients are integers. 

As an application we show that the simplest example of a $p$-adic KZ connection has an invariant line subbundle 
while its complex analog has no  nontrivial subbundles due to the irreducibility  of its monodromy representation.

\end{abstract}

{\small\tableofcontents\par}

\setcounter{footnote}{0}
\renewcommand{\thefootnote}{\arabic{footnote}}

\section{Introduction}

In the seminal work \cite{Dw} Dwork laid the foundation of the theory 
of $p$-adic hypergeometric differential equations.
His principal working example was the differential equation
\bean
\label{HE_}
x(1-x) I'' +(1-2x)I'-\frac14I=0,
\eean
whose analytic at the origin solution
\bean
\label{111_}
_2F_1\Big(\frac12,\frac 12; 1; x\Big) = \frac 1\pi\,\int_1^\infty t^{-1/2}(t-1)^{-1/2}(t-x)^{-1/2} dt
=\sum_{k = 0}^\infty\binom{-1/2}{k}^2x^k
\eean
encodes periods of the Legendre family $y^2=t(t-1)(t-x)$.
Dwork used the approximations
\bea
F_s(x)=\sum_{k = 0}^{p^s-1}\binom{-1/2}{k}^2x^k \qquad\text{for}\quad s=1,2,\dots,
\eea
which are nothing but truncations of the infinite sum in \eqref{111_} 
and clearly converge to it in the disk $D_{0,1}=\{x\mid |x|_p<1\}$,
to show that the uniform limit $F_{s+1}(x)/F_s(x^p)$ as $s\to\infty$ exists in a larger domain 
$\frak D^{\text{Dw}}$ and this limit, dubbed as the ``unit root'', corresponds to a root of the
 local zeta function of the $x$-fiber in the family.
Dwork's work boosted the whole body of research in the area; we limit ourselves to 
mentioning some recent contributions on the theme \cite{AS, BV, LTYZ}.

Interestingly enough, Dwork himself indicates in \cite{Dw} that he adopts a 
similar point of view in the $p$-adic case to the one Igusa had in~\cite{Ig} 
for modulo $p$ solution of \eqref{HE_}.
Namely, the cycles of the elliptic curve $y^2=t(t-1)(t-x)$ for a given $x$ can 
be thought of as the local at $x$ analytic  solutions of the differential equation
\eqref{HE_}.  At the same time, Igusa's modulo~$p$ solution
\bea
g(x)=\sum_{k = 0}^{(p-1)/2}\binom{(p-1)/2}{k}^2x^k
\eea
of \eqref{HE_}, though indeed coinciding with Dwork's $F_1(x)$ modulo $p$,
 hints at a somewhat different way for approximating \eqref{111_} $p$-adically
  through \emph{natural} truncations of the sum.
This recipe seems to escape its own development until very recently 
Schechtman and Varchenko gave solutions to general Knizhnik--Zamolodchikov (KZ) 
equations modulo~$p$ in \cite{SV2}, recovering Igusa's polynomial as a particular case.
It has been realized in \cite{V4} that this approach goes in parallel with that of Dwork in 
\cite{Dw}; explicitly, the polynomial solutions of hypergeometric equation and of a certain
 class of KZ equations modulo $p^s$ are shown in \cite{V4} to give rise to $p$-adic 
 approximations of the corresponding unit roots.
The principal goal of this paper is to prove such expectations, at least for the cases in
which the technicality of proofs do not overshadow the beauty of outcomes.

\vsk.2>

In this paper we study certain $p^s$-approximation polynomials of hypergeometric periods.
We consider an integral of hypergeometric type like in \eqref{HE_} without specifying the cycle of integration.
For any positive integer $s$ we replace the integrand by a polynomial $\Phi_{s}(t,x)$ with integer coefficients
called the master polynomial and  define the $p^s$-approximation polynomial as the coefficient of
$t^{p^s-1}$ in the master polynomial.
This is our $p^s$-analog of the initial integral.
 In the example of \eqref{HE_} the master polynomial is
$\Phi_{s}(t,x) = t^{(p^s-1)/2}(t-1)^{(p^s-1)/2}(t-x)^{(p^s-1)/2}$ and the $p^s$-approximation polynomial
is 
\bean
\label{P_s} 
P_s(x) = (-1)^{(p^s-1)/2} \sum_{k=0}^{(p^s-1)/2}\binom{(p^s-1)/2}{k}^2 x^k.
\eean
We prove the Dwork-type congruence,
\bean
\label{Pi} 
P_{s+1}(x) P_{s-1}(x^p) \equiv P_{s}(x) P_{s}(x^p) \pmod{p^s},
\eean
in Theorem \ref{conj B}. More general $p^s$-approximation 
constructions are discussed in \cite{SV2, V4, RV1, RV2}.

\vsk.2>  
In Section \ref{sec 7} we consider the simplest example of the KZ connection. The KZ connections are objects 
of conformal field theory, representation theory, enumerative geometry, see \cite{KZ, EFK, MO}.
In our example, the KZ connection is identified with the Gauss--Manin connection of
the family
of elliptic curves $y^2=(t-z_1)(t-z_2)(t-z_3)$. We study the $p^s$-approximation polynomials to the 
elliptic period $\int \!dt/y$ and show that the $p$-adic KZ connection of our example
has an invariant line subbundle. This is a  $p$-adic feature since the corresponding complex KZ connection
has no proper invariant subbundle due to the irreducibility of its monodromy representation.

The invariant subbundles of the KZ connection over $\C$ usually are related to some additional conformal block
constructions, for example see \cite{FSV, SV2, V3}. Apparently our subbundle   is of a
different $p$-adic nature.

\vsk.2>

The results above require proving $p$-adic convergence,
which in turn rests upon establishing certain special congruences.
Considering more general hypergeometric series $F(x)=\sum_{n=0}^\infty A(n)x^n$
and its $p^s$-truncations $F_s(x)=\sum_{n=0}^{p^s-1}A(n)x^n$,
Dwork showed that
\bean
\label{tI}
F_{s+1}(x) F_{s-1}(x^p) \equiv F_{s}(z) F_{s}(x^p) \pmod{p^s}
\quad\text{for}\; s=1,2,\dots,
\eean
in \cite[Theorem~2]{Dw}; this allowed him to conclude 
 the existence of the $p$-adic limit
 \linebreak
  $F_{s+1}(x)/F_s(x^p)$ as $s\to\infty$ in \cite[Theorem~3]{Dw}.
As an auxiliary component of Dwork's argument the other set of congruences, 
a l\`a Lucas, was used for the sequence of coefficients $A(n)$:
\bean
\label{tII}
\\[-5mm] \notag
\frac{A(n+mp^s)}{A([n/p]+mp^{s-1})}\equiv\frac{A(n)}{A([n/p])}\pmod{p^s}
\quad\text{for}\; m,n\in\Z_{\ge0} \;\text{and}\; s=1,2,\dots
\eean
(see Corollary~1(ii) on p.~36 in~\cite{Dw}).
These two different-looking families of congruences \eqref{tI} and \eqref{tII}
 are both known as Dwork congruences, and to distinguish between the two 
 we dub them type~I and type~II, respectively.
Our Igusa-inspired $p^s$-aproximations of solutions of hypergeometric equation, like \eqref{P_s} above, 
are of dual nature. Although congruences \eqref{Pi} look like congruences \eqref{tI} of type I
we may view the sequence $(P_s(x))_{s\geq 1}$ as a subsequence of a suitable
polynomial sequence $A(n;x)$ depending
 on the extra parameter $x$ and satisfying
\bean
\label{tII-x}
\\\notag
\frac{A(n+mp^s;x)}{A([n/p]+mp^{s-1};x^p)}\equiv\frac{A(n;x)}{A([n/p];x^p)}\pmod{p^s}
\qquad\text{for}\quad m,n\in\Z_{\ge0} \quad\text{and}\quad s=1,2,\dots\,.
\eean
Then the restriction of these type~II Dwork congruences to the 
subsequence $(P_s(x))_{s\geq 1}$ reads as
  type~I congruences for the $p^s$-aproximation polynomials 
in parameter $x$.
(We make this explicit in the remark after Theorem \ref{thm C_k}.)
To summarize, our principal tool for establishing the existence of $p$-adic 
convergence are Dwork-type congruences \eqref{tII-x}, of which
 required type~I congruences are particular instances.
General theorems towards Dwork's congruences of type~II
were given by Mellit in~\cite{Me} and, independently, by Samol and van Straten in~\cite{SvS}.
As these results are insufficiently general and only for \eqref{tII} rather than \eqref{tII-x},
we extend them further using the method from Mellit's unpublished preprint \cite{Me}
(see also \cite{MV} for another application).
Our result based on Mellit's elegant approach is displayed in Section~\ref{sec 2}, 
and its power is illustrated by the congruences \eqref{Pi} from Theorem \ref{conj B}
and by several other quite different applications in later sections.

In Section~\ref{sec 8} we conjecture some  stronger congruences for the polynomials
$P_s(x)$.

\subsection*{Acknowledgements}
The authors thank Frits Beukers, Andrew Granville, Anton Mellit, Rich\'ard Rim\'anyi, Steven Sperber, and Masha Vlasenko 
 for useful discussions.

\section{On ghosts}
\label{sec 2}

In this paper $p$ is an odd prime.

\subsection{Mellit's theorem}
 Let $\La(t)$ be a Laurent polynomial in variables $t=(t_1,\dots,t_r)$
with coefficients in $\Z_p$ and constant term $\on{CT}_t(\La)$. 
Assume that the Newton polytope  of $\La(t)$ contains only one interior point $\{0\}$.

\vsk.2>
For a tuple   $a=(a_0,a_1,\dots,a_{l-1})$, denote by $l(a)=l$ its length.
For two tuples $a$ and $b$, the concatenation product $a*b$ is the tuple
 of length $l(a)+l(b)$ obtained by gluing $a$ and $b$ together.
For  $a=(a_0$, \dots, $a_{l-1})$ of length $l$, denote by 
$a'=(a_1,\dots,a_{l-1})$ the ``derivative'' tuple of length $l-1$.
If  $a$ is a tuple of numbers, denote  $|a|=\sum_{i=0}^{l-1}a_i$.

\vsk.2>
For a tuple  $m=(m_0,\dots,m_l)$ of integers from $\{1,\dots,p-1\}$
denote by $\on{CT}_t(\La^m)$ the constant term of the Laurent polynomial
$\La(t)^{m_0+m_1p+m_2p^2+\dots+m_{l-1}p^{l-1}}$.

\begin{thm} [{Anton Mellit, 2009, \cite{Me}, unpublished}]
\label{thm Mellit}
Let $a,b,c$ be tuples  of integers from $\{1,\dots,p-1\}$, where $b,c,a'$ can be empty, that is, of length 0.
Then
\bean
\label{Me}
\on{CT}_t(\La^{a*b}) \on{CT}_t(\La^{a'*c}) \equiv \on{CT}_t(\La^{a'*b}) 
\on{CT}_t(\La^{a*c}) \pmod{p^{l(a)}}.
\eean

\end{thm}

We modify the statement and three-page Mellit's proof of Theorem \ref{thm Mellit}
and prove Theorem \ref{thm cg} below.

\subsection{Convex polytopes}
\label{sec A1}

Given a positive integer $r$ we consider convex polytopes, which are convex hulls of finite subsets of $\Z^r\subset \R^r$.

\begin{defn}
\label{defn}

A tuple $(N_0, N_1,\dots,N_{l-1})$ of convex polytopes is called $\on{admissible}$ if for any
$0\leq i\leq j \leq l-1$ we have
\bea
\big(N_i + pN_{i+1}+ \dots + p^{j-i}N_j\big)\cap p^{j-i+1}\Z^r = \{0\}.
\eea

\end{defn}

\subsection{Definition of ghosts}

Let $\La(t,z)$ be a Laurent polynomial in some variables 
$t=(t_1,\dots,t_r)$, $z=(z_1,\dots,z_{r'})$
with coefficients in $\Z_p$.  We 
define the ghost terms $R_m(\La)$, $m\geq 0,$
 as the unique sequence of Laurent polynomials  in $t,z$ satisfying the following two properties:
\begin{enumerate}[(i)]

\item \label{ghost-i}
For each $m$ we have
\bea
\La(t,z)^{p^m} = R_{0}(\La)(t^{p^{m}}, z^{p^{m}})
+R_{1}(\La)(t^{p^{m-1}}, z^{p^{m-1}})+\dots + R_{m}(\La)(t, z).
\eea

\item \label{ghost-ii}
For each $m$ the coefficients of $R_m(\La)(t,z)$ are divisible by $p^m$ in $\Z_p$.

\end{enumerate}

\vsk.2>
Properties \eqref{ghost-i}--\eqref{ghost-ii} recursively determine the polynomials $R_{m}(\La)(t, z)$.
Namely,
\bean
\label{R_m}
R_m(\La)(t,z) = \La(t,z)^{p^m} - \La(t^p,z^p)^{p^{m-1}}\,, \quad R_0(\La)(t,z) = \La(t,z). 
\eean

\vsk.2>
Let $F(t,z)$ be a Laurent polynomial in $t,z$. Let $N(F)$ be the Newton polytope of $F(t,z)$ with respect to 
the $t$ {\it variables only}.  Clearly, we have 
\bean
\label{subset}
N(R_m(\La))\subset p^mN(\La).
\eean

\subsection{Composed ghosts}

Let $\la = (\La_0(t,x), \dots , \La_{l-1}(t,z))$ be a tuple of Laurent polynomials with coefficients in $\Z_p$. 
We decompose the product 
\bea
\tilde \la(t,z) := \La_0(t,x) (\La_1(t,z))^p\cdots (\La_{l-1}(t,z))^{p^{l-1}}
\eea
into the sum of ghost terms of $\La_0, \dots, \La_{l-1}$. As the result we obtain that
$\tilde \la$ is the sum of the products
\bea
R_{m,\la}(t,z) 
&:=&
R_{m_0}(\La_0)(t,z)\cdot
R_{m_1}(\La_1)(t^{p^{1-m_1}},z^{p^{1-m_1}})
\cdot
R_{m_2}(\La_2)(t^{p^{2-m_2}},z^{p^{2-m_2}})
\cdots
\\
&&
\cdots R_{m_{l-1}}(\La_{l-1})(t^{p^{l-1-m_{l-1}}},z^{p^{l-1-m_{l-1}}}),
\eea
where $m=(m_0,\dots,m_{l-1})$ runs over the set of all $l$-tuples of integers satisfying $0\leq m_i \leq i$.
Clearly, we have
\bea
R_{m,\la}(t,z) \equiv 0 \pmod{p^{|m|}}
\eea
and
\bea
N(R_{m,\la}(t,z)) \subset N(\La_0(t,z))
+  pN(\La_1(t,z)) + \dots + p^{l-1} N(\La_{l-1}(t,z)).
\eea

\subsection{Indecomposable tuples}

Denote by $S_k$ the set of all $k$-tuples $m=(m_0,\dots,m_{k-1})$ of integers
such that $0\leq m_i\leq i$.
Put $S=\bigcup_{k=1}^\infty S_k$. A tuple $m\in S$ is called 
{\it indecomposable} if it cannot be presented as $m'*m''$ for
$m',m''\in S$. Denote by $S_k^{\on{ind}}$ the set of all indecomposable $k$-tuples and put 
$S^{\on{ind}} = \bigcup_{k=1}^\infty S_k^{\on{ind}}$.

\begin{lem}
\label{lem ind}
If $m\in S_k^{\on{ind}}$, then $|m|\geq k-1$.

\end{lem}

\begin{proof}
If $m$ is indecomposable, then for each $i\in\{1,\dots,k-1\}$ there exists
 $j\geq i$ such that $m_j>j-i$, i.e. $j\geq i>j-m_j$. The number 
 of such $i$ for a given $j$ is $m_j$. The total number of $i$ is $k-1$, therefore the sum of $m_j$ is at least $k-1$.
\end{proof}

\begin{lem}
\label{lem un dec}
For each $m\in S$, there exist unique indecomposable $m^1,\dots, m^r$ such that
$m=m^1*\dots *m^r$.
\end{lem}

\begin{proof} The proof is by induction on $l(m)$. If $l(m)=1$, then $m=(m_0)=(0)$ and $m$ is indecomposable.
Let us prove the induction step. Let 
\bea
m=m^1*\dots *m^r = n^1*\dots *n^s
\eea
be two decompositions into indecomposable factors. We may assume that $l(n^s)\geq l(m^r)$.
If $l(n^s) = l(m^r)$, then  $n^s = m^r$. In this case we can conclude that
$m^1*\dots *m^{r-1} = n^1*\dots *n^{s-1}$, and the statement follows from the induction hypothesis.
If $l(n^s)> l(m^r)$, then the sequence $n^s$ contains the sequence $m^r$ as its last
$l(m^r)$-part. This contradicts to the indecomposability of $n^s$. The lemma is proved.
\end{proof}

\subsection{Polynomials $I_\la$}

For an $l$-tuple $\la=(\La_0(t,z),\La_1(t,z),\dots,\La_{l-1}(t,z))$ of Laurent polynomials  with coefficients
in $\Z_p$ define
\bea
I_\la(t,z) = \sum_{m\in S_{l}^{\on{ind}}} R_{m,\la}(t,z).
\eea
We have
\bean
\label{equiv I}
I_\la(x,z) \equiv 0 \qquad \pmod{p^{l-1}}
\eean
by Lemma \ref{lem ind} and
\bea
N(I_{\la}(t,z)) \subset N(\La_0(t,z))
+  pN(\La_1(t,z)) + \dots + p^{l-1} N(\La_{l-1}(t,z))\,.
\eea

\begin{lem}
\label{lem decomp}
We have
\bean
\label{decomp}
&&
\\
&&
\notag
\tilde \la(t,z) =\sum_{\la=\la^1*\dots *\la^s}
I_{\la^1}(t,z)
I_{\la^2}(t^{p^{l(\la^1)}},z^{p^{l(\la^1)}})
\cdots
I_{\la^s}(t^{p^{l(\la^1)+\dots +l(\la^{s-1})}},z^{p^{l(\la^1)+\dots+l(\la^{s-1})}}),
\eean
where the sum is over the set of all possible decompositions of the tuple $\la$ into a product of tuples.
\end{lem}

\begin{proof} We have
\bea
\tilde \la(t,z) = \sum_{m\in S_{l}} R_{m,\la}(t,z).
\eea
For any $m\in S_{l}$, let $m=m^1*\dots *m^s$ be its unique idecomposition into
indecomposable factors. Let $\la=\la^1*\dots*\la^s$ be the corresponding factorization of
the sequence $\la$.  Then 
\bean
\label{ram}
{}
\\
\notag
R_{m,\la}(t,z) = R_{m^1,\la^1}(t,z) R_{m^2,\la^2}(t^{p^{l(\la^1)}},z^{p^{l(\la^1)}})
\cdots
R_{m^s,\la^s}(t^{p^{l(\la^1)+\dots+l(\la^{s-1})}},z^{p^{l(\la^1)+\dots+l(\la^{s-1})}}).
\eean
This product contributes to the the expansion of the product
\bean
\label{II}
I_{\la^1}(t,z)
I_{\la^2}(t^{p^{l(\la^1)}},z^{p^{l(\la^1)}})
\cdots
I_{\la^s}(t^{p^{l(\la^1)+\dots +l(\la^{s-1})}},z^{p^{l(\la^1)+\dots+l(\la^{s-1})}})
\eean
into the sum, and conversely each summand in the expansion of \eqref{II}
comes from \eqref{ram} for a unique indecomposable factorization $m=m^1*\dots*m^s$.
\end{proof}

\subsection{Admissible tuples of Laurent polynomials}

\begin{defn}
\label{defn2}
A tuple $\la=(\La_0(t,z),\La_1(t,z),\dots,\La_{l-1}(t,z))$ of Laurent polynomials 
is called $\on{admissible}$ if the tuple 
$(N(\La_0(t,z)), N(\La_1(t,z)),\dots,N(\La_{l-1}(t,z)))$ of its Newton polytopes with respect to variables $t$
is admissible.

\end{defn}

Denote by $\on{CT}_t(\La)(z)$ the constant term of the Laurent polynomial $\La(t,z)$ with respect to
the variables $t$.
The constant term $\on{CT}_t(\La)(z)$ is a Laurent polynomial in $z$.

\begin{lem}
\label{lem 1pt}

Let $\la=(\La_0(t,z),\La_1(t,z),\dots,\La_{l-1}(t,z))$  be an admissible tuple of Laurent polynomials
with coefficients in $\Z_p$ and $\la=\la^1*\dots *\la^s$. Then
\bean
\label{CII}
&&
\on{CT}_t\!\bigg(\prod_{i=1}^s
I_{\la^i}\big(t^{p^{l(\la^1)+\dots +l(\la^{i-1})}},z^{p^{l(\la^1)+\dots+l(\la^{i-1})}}\big)\bigg)\!(z)
=
\\
\notag
&&
\phantom{aaaaaaaaaaaaaaaa}
=\,\prod_{i=1}^s
\on{CT}_t \big(I_{\la^i}(t,z)\big)(z^{p^{l(\la^1)+\dots+l(\la^{i-1})}}).
\eean

\end{lem}

\begin{proof} 
We have
\bea
N(I_{\la^1}(t,z)) \subset 
N(\La_0(t,z))
+  pN(\La_1(t,z)) + \dots + p^{l(\la^1)-1} N(\La_{l(\la^1)-1}(t,z)).
\eea
Hence
\bea
N(I_{\la^1}(t,z))\cap p^{l(\la^1)}\Z^r = \{0\}
\eea
and
\bea
&&
\on{CT}_t\!\Big(\prod_{i=1}^s
I_{\la^i}(t^{p^{l(\la^1)+\dots +l(\la^{i-1})}},z^{p^{l(\la^1)+\dots+l(\la^{i-1})}})\Big)\!(z)
=
\\
&&
\phantom{aaaaaa}
= \on{CT}_t\big(I_{\la^1}(t,z)\big)(z)
\on{CT}_t\!\Big(
\prod_{i=2}^s
I_{\la^i}\big(t^{p^{l(\la^2)+\dots+l(\la^{i-1})}},z^{p^{l(\la^2)+\dots+l(\la^{i-1})}}\big)\Big)
\big(z^{p^{l(\la^1)}}\big).
\eea
Thus by induction on $s$ we prove the statement.
\end{proof}

\begin{cor}
\label{cor Cp}
We have
\bean
\label{Cp}
&&
\\
\notag
&&
\on{CT}_t\big(\tilde \la\big)(z)
 =\sum_{\la=\la^1*\dots *\la^s}
\on{CT}_t(I_{\la^1})(z)\cdot
\on{CT}_t(I_{\la^2})(z^{p^{l(\la^1)}})
\cdots
\on{CT}_t(I_{\la^s})(z^{p^{l(\la^1)+\dots+l(\la^{s-1})}}),
\eean
where the sum is over the set of all decompositions of $\la$ into a product of tuples.
\end{cor}

\subsection{Dwork congruence for tuples of Laurent polynomials}

\begin{thm}
\label{thm cg}
Let $a,b,c$ be tuples of  Laurent polynomials in $t,z$ with coefficients in $\Z_p$,
where $b,c,a'$ can be empty.
Assume that the tuples $a*b, a*c, a'*b, a'*c$ of Laurent polynomials are 
admissible. Then
\bean
\label{Da}
\phantom{aaaaaa}
\on{CT}_t\big(\widetilde{a*b}\big)\!(z)\, \on{CT}_t\big(\widetilde{a'*c}\big)\!(z^p)
\,\equiv\,
\on{CT}_t\big(\widetilde{a'*b}\big)\!(z^p) \,\on{CT}_t\big(\widetilde{a*c}\big)\!(z)
\pmod{p^{l(a)}}.
\eean

\end{thm}

\begin{proof} The left-hand side  and right-hand  side of \eqref{Da} are
\bean
\label{LHS}
\sum_{{a*b=x^1*\dots *x^q\atop a'*c=y^1*\dots *y^s}}
\prod_{i=1}^q
\on{CT}_t(I_{x^i})(z^{p^{l(x^1)+\dots+l(x^{i-1})}})
\prod_{i=1}^s
\on{CT}_t(I_{y^j})(z^{p^{1+l(y^1)+\dots+l(y^{j-1})}})
\eean
and
\bean
\label{RHS}
\sum_{{a'*b=x^1*\dots *x^q\atop a*c=y^1*\dots *y^s}}
\prod_{i=1}^q
\on{CT}_t(I_{x^i})(z^{p^{1+l(x^1)+\dots+l(x^{i-1})}})
\prod_{i=1}^s
\on{CT}_t(I_{y^j})(z^{p^{l(y^1)+\dots+l(y^{j-1})}}),
\eean
respectively.
Since we work modulo $p^{l(a)}$, all the terms with
\bea
\sum_{i=1}^q l(x^i) + \sum_{j=1}^s l(y^j) -q-s\geq l(a)
\eea
may be dropped off from consideration.
That inequality can be reformulated as
$l(a)+l(b)+l(a)+l(c)-1 - q-s \geq l(a)$, equivalently, as
\bean
\label{ineq}
l(a)+l(b)+l(c) \geq q+s+1.
\eean

Let us prove that the remaining terms in both expressions are in a bijective 
correspondence such that the corresponding terms are \emph{equal}.

\vsk.2>

Namely, take  one of the remaining summands on the left-hand side:
\bean
\label{rem s}
\prod_{i=1}^q
\on{CT}_t(I_{x^i})(z^{p^{l(x^1)+\dots+l(x^{i-1})}})
\prod_{i=1}^s
\on{CT}_t(I_{y^j})(z^{p^{1+l(y^1)+\dots+l(y^{j-1})}}),
\eean
the summand corresponding to the presentation
 $a*b=x^1*\dots *x^q$, $a'*c=y^1*\dots *y^s$.

\begin{lem}
\label{lem ij} There exist indices $i\geq 1$ and $j\geq 0$ such that
\bean
\label{iiS}
l(x^1)+\dots+l(x^i) = l(y^1)+\dots+l(y^j) + 1\leq l(a).
\eean

\end{lem}

\begin{proof}
If $l(x^1)=1$, then $i=1$ and $j=0$ are the required indices.

Assume that $l(x^1)>1$ and the required  $i$, $j$ do not exist.
Then each number in $\{2, \dots,l(a)\}$
cannot be represented simultaneously as 
$l(x^1)+\dots+l(x^i)$ and $l(y^1)+\dots+l(y^j) + 1$.
Therefore the sum of the total number of $i\geq 1$, such that 
$l(x^1)+\dots+l(x^i)\leq l(a)$, and the total number of
$j\geq 1$, such that $l(y^1)+\dots+l(y^j) + 1\leq l(a),$
is at most $l(a)-1$. The number of remaining $i$ 
is at most $l(b)$ and the number of remaining $j$ is at most $l(c)$.
 Therefore, $q+s \leq l(a) - 1 + l(b) + l(c)$, which is the same as \eqref{ineq}. Hence  
 the corresponding summand must have been dropped off. This establishes the
  existence of indices $i$ and $j$ required.
\end{proof}

Now we return to the remaining summand \eqref{rem s}.  
Choose the minimal indices $i\geq 1$ and $j\geq 0$ such that
\eqref{iiS} holds. Then it is easy to see that
\bean
\label{nde}
\phantom{aaa}
a'*b=y^1*\dots *y^j *x^{i+1}*\dots *x^q, \qquad
a*c =x^1*\dots *x^i * y^{j+1}*\dots *x^s,
\eean
and the summand in \eqref{RHS} corresponding to the presentations in \eqref{nde} equals
the product in \eqref{rem s}. This clearly gives the desired bijection.
\end{proof}

\section{$p^s$-Approximation of a hypergeometric integral}
\label{sec 3}

Let $\al,\beta, \ga$ be rational numbers with $|\al|_p=|\beta|_p=|\ga|_p=1$. 
Consider a hypergeometric integral
\bean
\label{I ga}
I^{(C)}(x) =\int_C t^\al(t-1)^\beta (t-x)^\ga dt
\eean
where  $C\subset \C-\{0,1,x\}$  is a contour on which the integrand  takes its initial value when $t$ encircles $C$.
The function $I^{(C)}(x)$ satisfies the hypergeometric differential equation
\bean
\label{hye}
x(1-x) I'' +((\al+\beta+2\ga)x - (\al+\ga))I' -\ga(\al+\beta+\ga+1) I=0.
\eean
This follows from Stokes' theorem and the following identity of differential forms. Denote
$\Phi(t,x) = t^\al(t-1)^\beta (t-x)^\ga$,
\bea
\mc D = 
x(1-x) \frac{d^2}{d x^2} +((\al+\beta+2\ga)x - (\al+\ga))\frac{d}{d x} 
-\ga(\al+\beta+\ga+1).
\eea
 Then
\bean
\label{Stokes}
&&
d_t\Big(\ga \frac{t(t-1)}{t-x}\Phi(t,x)\Big)
= \mc D \, \Phi(t,x) dt.
\eean

The differential equation \eqref{hye} turns into the standard hypergeometric differential equation
\bean
\label{sde}
x(1-x)I''+(c-(a+b+1)x)I'-abI=0
\eean
if  $\al=a-c$, $\beta=c-b-1$, $\ga=-a$. For a suitable choice of $C$ and multiplication of the integral  by a constant,
the integral in \eqref{I ga} can be expanded as a power series
$$
_2F_1(a,b;c;x)
={}_2F_1\biggl(\begin{matrix} a, \, b \\ c \end{matrix}\biggm| x\biggr)
=\sum_{k=0}^\infty\frac{(a)_k(b)_k}{k!(c)_k}x^k\,.
$$
Here 
$(a)_n=\Gamma(a+n)/\Gamma(a)=\prod_{k=0}^{n-1}(a+k)$ stands for Pochhammer's symbol.

We consider the following $p^s$-approximation of the integral in \eqref{I ga}. 
Given a positive integer $s$, let $1\leq \al_s, \beta_s, \ga_s \leq p^s$ be the unique positive integers
such that 
\bean
\label{eq abg}
\al_s\equiv \al, \quad
\beta_s\equiv \beta, \quad \ga_s\equiv\ga \qquad \pmod{p^s}.
\eean
Define the {\it master polynomial}
\bean
\label{Phis}
\Phi_s(t,x) = t^{\al_s}(t-1)^{\beta_s} (t-x)^{\ga_s}
\eean
and the {\it $p^s$-approximation polynomial} $I_s(x)$ as the coefficient of $t^{p^s-1}$
in the master polynomial $\Phi_s(t,x)$. Then
\bean
\label{Is ap}
I_s(x) = (-1)^{\al_s+\beta_s+\ga_s-p^s+1}\sum_{k_1+k_2=\al_s+\beta_s+\ga_s-p^s+1}
\binom{\beta_s}{k_1}\binom{\ga_s}{k_2}x^{k_2} .
\eean
The polynomial $I_s(x)$ has integer coefficients.

\begin{thm}
\label{thm m}
The polynomial $I_s(x)$ is a solution of the hypergeometric differential equation
\eqref{hye} modulo $p^s$,
\bean
\label{s appr}
\mc D I_s(x) \,\in\, p^s \Z_p[x]\,.
\eean
\end{thm}

\begin{proof}
The theorem follows from formula \eqref{Stokes}.
\end{proof}

In this paper we prove Dwork-type congruences for the $p^s$-approximation polynomials $I_s(x)$
in several basic examples and leave general considerations for another occasion.

\vsk.2>
For more general versions of the $p^s$-approximation construction see in \cite{SV2, V4}.

\section{Function $_2F_1\Big(\frac12,\frac 12; 1,x\Big)$}
\label{sec 4}

\subsection{Polynomials $P_s(x)$}
\label{sec 1.1}

The function
\bean
\label{111}
_2F_1\Big(\frac12,\frac 12; 1; x\Big) = \frac 1\pi\,\int_1^\infty t^{-1/2}(t-1)^{-1/2}(t-x)^{-1/2} dt
=\sum_{k = 0}^\infty\binom{-1/2}{k}^2x^k
\eean
satisfies the hypergeometric differential equation 
\bean
\label{HE}
x(1-x) I'' +(1-2x)I'-\frac14I=0.
\eean
 Define the master polynomial
\bean
\label{mp1/2}
\Phi_{p^s}(t,x) = t^{(p^s-1)/2}(t-1)^{(p^s-1)/2}(t-x)^{(p^s-1)/2}.
\eean

\vsk.2>
\noindent
The number 
$M=\frac{p^s-1}2=\frac{p-1}2+\frac{p-1}2 p+\dots +\frac{p-1}2 p^{s-1}$
 is the unique positive integer such that $1\leq M\leq p^s$ and 
$M\equiv -1/2$ $\pmod{p^s}$.
Define the $p^s$-approximation polynomial $P_s(x)$ as
the coefficient of $t^{p^s-1}$ in the master polynomial $\Phi_{p^s}(t,x)$. Then
\bean
\label{Ps}
P_s(x) = (-1)^{(p^s-1)/2} \sum_{k=0}^{(p^s-1)/2}\binom{(p^s-1)/2}{k}^2 x^k.
\eean
Define $P_0(x)=1$.

Recall the hypergeometric function ${}_2F_1(a,b;c;x)$.  Then
\bean
\label{p hy}
P_s(x) = (-1)^{(p^s-1)/2} {}_2F_1\Big(\frac{1-p^s}2, \frac{1-p^s}2; 1; x\Big).
\eean
The polynomial $P_s(x)$ is a solution of the hypergeometric equation \eqref{HE} modulo $p^s$. This 
follows from Theorem  \ref{thm m} or  from formula \eqref{p hy}.

\subsection{Baby congruences}

Let
$\phi_s(x)=(x+1)^{(p^s-1)/2}$. Then
\bean
\label{bin}
\phi_{s+1}(x)\phi_{s-1}(x^p)\equiv\phi_s(x)\phi_s(x^p)\pmod{p^s}.
\eean
This follows from  $(x+1)^{p^s}\equiv(x^p+1)^{p^{s-1}}\pmod{p^s}$.

\begin{lem}
\label{lem mm=mm}  The master polynomials $\Phi_{s}(t,x)$ satisfy 
the baby congruence
\bean
\label{mm=mm}
\Phi_{s+1}(t,x)\Phi_{s-1}(t^p,x^p) \equiv 
\Phi_{s}(t,x)\Phi_{s}(t^p,x^p) \pmod{p^s}.
\eean
\end{lem}

\begin{proof}
The lemma follows from \eqref{bin}. 
\end{proof}

\subsection{Congruences for $P_s(x)$}

\begin{thm}
\label{conj B}
\label{conj Q}

The approximation polynomials $P_s(x)$ satisfy the congruence
\bean
\label{CO}
P_{s+1}(x) P_{s-1}(x^p) \equiv P_{s}(x) P_{s}(x^p) \pmod{p^s}.
\eean
\end{thm}

This theorem follows from a more general Theorem \ref{thm C_k} below.

\vsk.2>

Using formula \eqref{p hy} we may rewrite \eqref{CO} as the congruence
\bean
\label{pF}
&
{}_2F_1\Big(\frac{1-p^{s+1}}2, \frac{1-p^{s+1}}2; 1; x\Big)
{}_2F_1\Big(\frac{1-p^{s-1}}2, \frac{1-p^{s-1}}2; 1; x^p\Big)
\equiv
\phantom{aaaaaaaaaa}
\\
\notag
&
\equiv
{}_2F_1\Big(\frac{1-p^s}2, \frac{1-p^s}2; 1; x\Big){}_2F_1\Big(\frac{1-p^s}2, \frac{1-p^s}2; 1; x^p\Big)
\pmod{p^s}.
\eean

Let $\al$ be a rational number which is a $p$-adic unit,
$\al = \al_0 + \al_1 p+ \al_2p^2+\cdots$.
Denote by $[\al]_s$ the sum of the first $s$ summands. Then congruence \eqref{pF} takes the form:

\bean
\label{-1/2s}
&
{}_2F_1\big([-\frac12]_{s+1}, [-\frac12]_{s+1};1;x\big)\,
{}_2F_1\big([-\frac12]_{s-1}, [-\frac12]_{s-1};1;x^p\big)
\equiv \phantom{aaaaaaaaaaaa}
\\
\notag
&
\equiv
{}_2F_1\big([-\frac 12]_{s}, [-\frac12]_{s};1;x\big)\,
{}_2F_1\big([-\frac12]_{s}, [-\frac12]_{s};1;x^p\big)
\pmod{p^s}.
\eean

\smallskip

\subsection{Coefficients of master polynomials}

Consider
\bean
\label{hat-Phi}
\hat\Phi_s(t,x)
&:=& t^{-(p^s-1)}\Phi_s(t,x)
= t^{-(p^s-1)/2}\big((t-1)(t-x)\big)^{(p^s-1)/2} 
\\ 
\nonumber
&=& \big((t-1)(1-x/t)\big)^{(p^s-1)/2}
=\sum_{j=-(p^s-1)/2}^{(p^s-1)/2} C_{s,j}(x)t^j,
\eean
where
\bea
C_{s,j}(x) = (-1)^{\frac{p^s-1}2-j}\sum_{m} \binom{\frac{p^s-1}2}{m+j}
 \binom{\frac{p^s-1}2}{m} x^m.
 \eea
In particular,  $C_{s,0}(x) = P_s(x)$.  Every coefficient $C_{s,j}(x)$ is a hypergeometric function:
\bea
C_{s,j}(x) = (-1)^{\frac{p^s-1}2-j} \binom{\frac{p^s-1}2}{j}
 {}_2F_1\Big( \frac{1-p^s}2, \frac{1-p^s}2+j; j+1; x\Big) \quad\text{for}\; j\ge0,
\eea
 while a hypergeometric expression in the case $j<0$ comes out from the following simple fact.

\begin{lem}
\label{lem ptp}
We have
$\hat\Phi_s(x/t,x) = \hat\Phi_s(t,x)$ and hence
\bean
\label{rel0}
C_{s,-j}(x) = x^j C_{s,j}(x).
\eean
\end{lem}

We expand congruence \eqref{mm=mm} into a congruence of 
polynomials $C_{s,j}(x)$. The constant term in $t$ gives us 
\bean
\label{c=c}
{\sum}_{k} C_{s+1,kp}(x)C_{s-1,-k}(x^p)
\equiv {\sum}_{k} C_{s,kp}(x)C_{s,-k}(x^p)
\pmod{p^s}.
\eean

\vsk.2>
The following Theorem \ref{thm C_k} establishes the congruences of individual pairs of terms
in \eqref{c=c}.

\begin{thm}
\label{thm C_k} For any $k$ appearing in \eqref{c=c} we have
\bean
\label{ex}
C_{s+1,kp}(x)C_{s-1,-k}(x^p) \equiv C_{s,kp}(x)C_{s,-k}(x^p) \pmod{p^s}.
\eean
In particular for $k=0$ we have congruence \eqref{CO}.
\end{thm}

\begin{proof}
Every index $k$ appearing in \eqref{c=c} can be written uniquely as 
\bean
\label{kf}
\phantom{aaa}
k= k_0 + k_1p+\dots + k_{s-2}p^{s-2}, 
\quad
-(p-1)/2\leq k_i \leq (p-1)/2.
\eean

Using \eqref{rel0} we  reformulate \eqref{ex} as
\bean
\label{exp0}
C_{s+1,kp}(x)C_{s-1,k}(x^p) \equiv C_{s,kp}(x)C_{s,k}(x^p) \mod{p^s}
\eean
and prove it below.  We have
\bea
C_{s+1,kp}(x) 
&=& 
\on{CT}_t\Big[\hat\Phi_1(t,x)
\Big(\prod_{i=0}^{s-2} \big( t^{-k_i}\hat\Phi_1(t,x)\big)^{p^{i+1}}\Big) \hat\Phi_1(t,x)^{p^s}\Big],
\\
C_{s-1,k}(x) 
&=& 
\on{CT}_t\Big[
\prod_{i=0}^{s-2} \big( t^{-k_i}\hat\Phi_1(t,x)\big)^{p^{i}}\Big],
\\
C_{s,kp}(x) 
&=& 
\on{CT}_t\Big[\hat\Phi_1(t,x)
\prod_{i=0}^{s-2} \big( t^{-k_i}\hat\Phi_1(t,x)\big)^{p^{i+1}}\Big],
\\
C_{s,k}(x) 
&=& 
\on{CT}_t\Big[
\Big(\prod_{i=0}^{s-2} \big( t^{-k_i}\hat\Phi_1(t,x)\big)^{p^{i}}\Big) \hat\Phi_1(t,x)^{p^{s-1}}\Big].
\eea
It is easy to see that the $(s+1)$-tuple of Laurent polynomials
\bea
\hat\Phi_1(t,x),\, t^{-k_0}\hat\Phi_1(t,x), \,\dots, \,t^{-k_{s-2}}\hat\Phi_1(t,x),\, \hat\Phi_1(t,x)
\eea
is admissible in the sense of Definition  \ref{defn2}. Now the application of Theorem \ref{thm cg} gives congruence 
\eqref{exp0} and hence congruence \eqref{ex}.
\end{proof}

\begin{rem}
Denote  $A(n,x) :={}_2F_1(-n,-n;1;x) = \sum_k  \binom{n}{k}^2 x^k$.
Let 
\bea
&
n = n_0+n_1p+ \dots + n_{s-1}p^{s-1},
\qquad
[n/p] = n_1 +\dots + n_{s-1}p^{s-2},
\eea
where $0\leq n_i<p$. Then for any $m\in\Z_{\geq 0}$ we have
\bean
\label{c An}
A(n +mp^s, x) A([n/p],x^p) \equiv A(n,x) A([n/p]+mp^{s-1},x^p) \pmod{p^s}.
\eean
\vsk.2>
\noindent
The proof follows from Theorem \ref{thm cg} and the identity
\bea
A(n,x) = \on{CT}_t\big[\big((t+1)(1+x/t)\big)^n\big].
\eea

\end{rem}

\subsection{Limits of $P_s(x)$}

For  $\al\in\Z_p$ there exists a unique solution $\om(\al)\in \Z_p$ of the equation 
$\om(\al)^p=\om(\al)$ that is congruent to $x$ modulo $p$. The element $\om(\al)$
is called the Teichm\"uller representative of $\al$. 
For $\al \in\F_p$, $r>0$, define the disc
\bea
 D_{\al,r} = \{ x\in \Z_p \mid  |x-\om(\al)|_p < r \}.
\eea
Denote
\bea
\bar P_s(x) := (-1)^{(p^s-1)/2} P_s(x) = {}_2F_1\Big(\Big[\!-\frac 12\Big]_{s},\Big[\!-\frac12\Big]_{s};1;x\Big),
\eea
see \eqref{p hy}.
Denote
\bean
\label{fD P}
 \frak D = \{x\in \Z_p \mid |\bar P_1(x)|_p=1\}.
\eean

\begin{thm}
\label{thm LP}

For $s\geq 1$ the rational function $\frac{\bar P_{s+1}(x)}{\bar P_s(x^p)}$  is regular on $\frak D$.
The sequence  $\big(\frac{\bar P_{s+1}(x)}{\bar P_s(x^p)}\big)_{s\geq 1}$ uniformly converges on $\frak D$.
The limiting analytic function $f(x)$ equals the ratio
$\frac{F(x)}{F(x^p)}$  on the disc $D_{0,1}$ where 
$F(x):={}_2F_1(1/2,1/2;1;x)$ is defined by the convergent power series  \eqref{111}.
We also have $|f(x)|_p=1$ for $x\in \frak D$.

\end{thm}

\begin{proof}
We have $\Z_p = \bigcup_{\al\in\F_p}D_{\al,1}$  and also
$\frak D= \bigcup_{\al\in\F_p,\, |\bar P_1(\om(\al))|_p =1} D_{\al,1}$
since $\bar P_1(x)$ has coefficients in $\Z_p$.  In particular, $D_{0,1}\subset \frak D$.
We also have
$$
\{x\in \Z_p \mid |\bar P_1(x^p)|_p=1\} = \bigcup_{\al\in\F_p,\, |\bar  P_1(\om(\al))|_p =1} D_{\al,1} = \frak D
$$
for the same reason.

By Lucas' theorem  $\bar P_s(x) \equiv \bar P_1(x) \bar P_1(x^p)\dots \bar  P_1(x^{p^{s-1}})$ $\pmod{p}$.
 Hence
$|\bar  P_s(x)|_p = |\bar  P_s(x^p)|_p=1$ for $s\geq 1$, $x\in\frak D$. Hence the rational 
functions $\frac{\bar P_{s+1}(x)}{\bar P_s(x^p)}$ are regular on $\frak D$.

Congruence \eqref{CO} implies that
\bean
\label{COr}
\Big|\frac{\bar P_{s+1}(x)}{\bar P_{s}(x^p)} - \frac{\bar P_{s}(x)}{\bar P_{s-1}(x^p)}\Big|_p\leq p^{-s}
\quad\on{for}\ x\in\frak D.
\eean
This shows the uniform convergence of our sequence of rational functions on the domain $\frak D$.
For the limiting function $f(x)$ we have  $|f(x)|_p=1$ for $x\in \frak D$.

Clearly, for any fixed index $k$ the coefficient $\binom{(p^s-1)/2}{k}^2$ of $x^k$ in $\bar P_s(x)$
converges $p$-adically  to the coefficient $\binom{-1/2}{k}^2$ of $x^k$ in $F(x)$.
Hence  the sequence $(\bar P_s(x))_{s\geq 1}$ converges to $F(x)$ on $D_{0,1}$, so that
$f(x) =\frac{F(x)}{F(x^p)}$  on  $D_{0,1}$.
The theorem is proved.
\end{proof}

 Dwork gives in \cite{Dw} a different construction of analytic continuation of the
ratio $\frac{F(x)}{F(x^p)}$ 
 from $D_{0,1}$ to a larger domain. He considers the sequence of polynomials
\bea
F_s(x) = \sum_{k=0}^{p^s-1}\binom{-1/2}{k}^2 x^k,
\eea
which are truncations of the hypergeometric series $F(x)$, and  shows that the sequence of
rational functions $\big(\frac{F_{s+1}(x)}{F_s(x^p)}\big)_{s\geq 1}$ uniformly converges on the domain
$\frak D^{\on{Dw}} =\{x\in\Z_p\ |\ |g(x)|_p=1\}$, where the polynomial
\bean
\label{Igu} 
g(x) =\sum_{k=0}^{(p-1)/2} \binom{-1/2}{k}^2x^k
\eean
is attributed by Dwork to Igusa \cite{Ig}. Clearly his limiting function $f^{\on{Dw}}(x)$ equals the ratio
$\frac{F(x)}{F(x^p)}$
on  $D_{0,1}$.

\vsk.2>

It is easy to see that the two sequences of rational functions  $\big(\frac{\bar P_{s+1}(x)}{\bar P_s(x^p)}\big)_{s\geq 1}$
and $\big(\frac{F_{s+1}(x)}{F_s(x^p)}\big)_{s\geq 1}$
 have the same limiting functions on the same domain. Indeed, $\bar P_1(x)\equiv g(x)$ $\pmod{p}$  and 
hence $\frak D=\frak D^{\on{Dw}}$. Also $f(x) = f^{\on{Dw}}(x)$ 
on $D_{0,1}$ and hence on $\frak D$.

\vsk.2>
Dwork shows in \cite{Dw} interesting properties of the function $f(x)$. For example,
let $\al \in \F^\times_p-\{1\}$ be such that $\om(\al) \in \fD$. Dwork shows  that 
the zeta function of  the elliptic curve defined over
$\F_p$  by the equation
$y^2=x(x-1)(x-\al)$ has two zeros, which are
$1/((-1)^{(p-1)/2} f(\om(\al)))$ and $(-1)^{(p-1)/2} f(\om(\al))/p$.
Clearly  $|f(\om(\al))|_p=1$.
The number $(-1)^{(p-1)/2} f(\om(\al))$ is 
called the unit root of that elliptic curve.

\vsk.2>

According to our discussion this unit root can be calculated as the value at $x=\om(\al)$ of the limit as $s\to\infty$
of the ratio $\frac{\bar P_{s+1}(x)}{\bar P_s(x^p)}$ of approximation polynomials 
multiplied by $(-1)^{(p-1)/2}$.

\section{Function $_2F_1\Big(\frac23,\frac 13; 1;x\Big)$}
\label{sec 5}

\subsection{Two hypergeometric integrals}

The  integral
\bean
\label{211}
I^{(C)}(x) =\int_C t^{-1/3}(t-1)^{-1/3} (t-x)^{-2/3} dt
\eean
where  $C\subset \C-\{0,1,x\}$  is a contour on which the integrand  takes its initial value when $t$ encircles $C$
satisfies the hypergeometric differential equation
\bean
\label{HE/3}
x(1-x) I'' +(1-2x)I'-\frac29I=0. 
\eean
For a suitable choice of $C$ the integral $I^{(C)}(x)$ presents the hypergeometric function
\bean
\label{2}
_2F_1\Big(\frac23,\frac 13; 1,x\Big) = \sum_{k = 0}^\infty\binom{-1/3}{k}\binom{-2/3}{k}x^k\,.
\eean

The  integral
\bean
\label{121}
J^{(D)}(x) =\int_D t^{-2/3}(t-1)^{-2/3} (t-x)^{-1/3} dt
\eean
where  $D\subset \C-\{0,1,x\}$  is a contour on which the integrand  takes its initial value when $t$ encircles $D$
satisfies the same hypergeometric differential equation.
For a suitable choice of $D$ the integral $J^{(D)}(x)$ presents the same hypergeometric function
$_2F_1\Big(\frac23,\frac 13; 1,x\Big)$. 

\vsk.2>
The differential form
$t^{-1/3}(t-1)^{-1/3}(t-z)^{-2/3} dt$ is transformed to the differential form
 $-t^{-2/3}(t-1)^{-2/3}(t-z)^{-1/3} dt$ by the change of the variable
$t\mapsto (t-z)/(t-1)$.

\vsk.2>

In this section we discuss the $p^s$-approximations of the integrals
$I^{(C)}(x)$ and $J^{(D)}(x)$.

\subsection{The case $p=3\ell+1$}
The master polynomial for $I^{(C)}(x)$ is given by the formula
\bea
\Phi_s(t,x) = t^{(p^s-1)/3}(t-1)^{(p^s-1)/3}(t-x)^{2(p^s-1)/3}.
\eea
The $p^s$-approximation  polynomial $Q_s(x)$ is defined as the coefficient of 
$t^{p^s-1}$ in $\Phi_s(t,x)$, 
\bea
Q_s(x) = (-1)^{(p^s-1)/3}\sum_{k}  \binom{2(p^s-1)/3}{k} \binom{(p^s-1)/3}{k} x^k.
\eea
Define $Q_0(x)=1$.  We have
\bean
\label{Q hy}
Q_s(x) =  {}_2F_1\Big(\frac{2-2p^s}3, \frac{1-p^s}3; 1; x\Big),
\eean
since $(-1)^{(p^s-1)/3}=1$.

\vsk.2>
The polynomial $Q_s(x)$ is a solution of the hypergeometric equation \eqref{HE/3} modulo $p^s$. This 
follows from Theorem  \ref{thm m} or  from formula \eqref{Q hy}.

\vsk.2>
The master polynomial for $J^{(D)}(x)$ is given by the formula
\bea
\Psi_s(t,x) = t^{2(p^s-1)/3}(t-1)^{2(p^s-1)/3}(t-x)^{(p^s-1)/3}.
\eea
The $p^s$-approximation polynomial $R_s(x)$ is defined as the coefficient of 
$t^{p^s-1}$ in $\Psi_s(t,x)$, 
\bea
R_s(x) = \sum_{k}  \binom{2(p^s-1)/3}{k} \binom{(p^s-1)/3}{k} x^k.
\eea
Define $R_0(x)=1$.  We have 
\bean
\label{Q=R}
Q_s(x) = {}_2F_1\Big(2\frac{1-p^s}3, \frac{1-p^s}3; 1; x\Big).
\eean

Master polynomials satisfy baby congruences,
\bean
\label{bc 1}
&&
\Phi_{s+1}(t,x) \Phi_{s-1}(t^p,x^p) \equiv \Phi_{s}(t,x) \Phi_{s}(t^p,x^p) \pmod{p^s}, 
\\
&&
\notag
\Psi_{s+1}(t,x) \Psi_{s-1}(t^p,x^p) \equiv \Psi_{s}(t,x) \Psi_{s}(t^p,x^p) \pmod{p^s} ,
\eean
by formula \eqref{bin}.

\begin{thm}
\label{th-QR1}
For $p=3\ell+1$ the approximation polynomials $R_s(x)$ and $Q_s(x)$ satisfy the congruences
\bean
\label{QQ}
&&
Q_{s+1}(x) Q_{s-1}(x^p) \equiv Q_{s}(x) Q_{s}(x^p) \pmod{p^s},
\\
\label{RR}
&&
R_{s+1}(x) R_{s-1}(x^p) \equiv R_{s}(x) R_{s}(x^p) \pmod{p^s}.
\eean
\end{thm}

\begin{proof} 
Denote $\hat\Phi_1(t,x) = (t-1)^{\ell}(1-x/t)^{2\ell}$. Then
$Q_{s+1}(x) = \on{CT}_t\big[\hat\Phi_1(t,x)^{1+p+\dots+p^s}\big]$.
It is easy to see that the $(s+1)$-tuple of Laurent polynomials
$(\hat\Phi_1(t,x),\hat\Phi_1(t,x), \,\dots, \hat\Phi_1(t,x))$
is admissible in the sense of Definition  \ref{defn2}. 
Now the application of Theorem \ref{thm cg} gives congruence 
\eqref{QQ}.
Congruence \eqref{RR} is proved in the same way applied to 
the formula  $R_{s+1}(x)$ $=$ $\on{CT}_t\big[\hat\Psi_1(t,x)^{1+p+\dots+p^s}\big]$, where
$\hat\Psi_1(t,x) = (t-1)^{2\ell}(1-x/t)^{\ell}$.
Congruence \eqref{RR} also follows from \eqref{QQ} since $R_s(x)=Q_s(x)$.
\end{proof}

Formulas \eqref{Q hy} and \eqref{Q=R} imply that for $p=3\ell+1$,\, $s\geq 1$ we have  
\bean
\label{3l+1}
&{}_2F_1\Big(\frac{2-2p^{s+1}}3, \frac{1-p^{s+1}}3; 1; x\Big)
{}_2F_1\Big(\frac{2-2p^{s-1}}3, \frac{1-p^{s-1}}3; 1; x^p\Big) 
\equiv \phantom{aaaaaaaaaaaaaaaaaa}
\\
\notag
&
\equiv
{}_2F_1\Big(\frac{2-2p^s}3, \frac{1-p^s}3; 1; x\Big)
{}_2F_1\Big(\frac{2-2p^s}3, \frac{1-p^s}3; 1; x^p\Big) \pmod{p^s}.
\eean
Using the expansions
\begin{alignat*}{2}
-1/3 &= \ell + \ell p+\ell p^2+\cdots ,
&\qquad
-2/3 &= 2\ell + 2\ell p+2\ell p^2+\cdots ,
\\
(p^s-1)/3 &= \ell + \ell p+\ell p^2+\dots  + \ell p^{s-1},
&\qquad
(p^s-2)/3 &= 2\ell + 2\ell p+2\ell p^2+\dots  + 2\ell p^{s-1},
\end{alignat*}
we conclude that for $p=3\ell+1$ and $s\geq 1$ we have

\bean
\label{132+1} 
&
{}_2F_1\big([-\frac23]_{s+1}, [-\frac13]_{s+1};1;x\big)\,
{}_2F_1\big([-\frac23]_{s-1}, [-\frac13]_{s-1};1;x^p\big)
\equiv \phantom{aaaaaaaaaaaa}
\\
\notag
&
\equiv
{}_2F_1\big([-\frac 23]_{s}, [-\frac13]_{s};1;x\big)\,
{}_2F_1\big([-\frac23]_{s}, [-\frac13]_{s};1;x^p\big)
\pmod{p^s}.
\eean

\vsk.2>

\subsection{The case $p=3\ell+2>2$}
The master polynomial for $I^{(C)}(x)$ is given by the formulas
\bea
\Phi_s(t,x) &=& t^{(2p^s-1)/3}(t-1)^{(2p^s-1)/3}(t-x)^{(p^s-2)/3}, \qquad \!\!\on{odd}\,s,
\\
\Phi_s(t,x) &=& t^{(p^s-1)/3}(t-1)^{(p^s-1)/3}(t-x)^{2(p^s-1)/3}, \qquad \on{even}\,s.
\eea
The $p^s$-approximation polynomial $Q_s(x)$ is defined as the coefficient of 
$t^{p^s-1}$ in $\Phi_s(t,x)$, 
\bea
Q_s(x) &=& (-1)^{(2p^s-1)/3}\sum_{k}  \binom{(2p^s-1)/3}{k} \binom{(p^s-2)/3}{k} x^k,
 \qquad \!\!\on{odd}\,s,
\\
Q_s(x) &=& (-1)^{(p^s-1)/3}\sum_{k}  \binom{2(p^s-1)/3}{k} \binom{(p^s-1)/3}{k} x^k,
 \qquad \on{even}\,s.
\eea
Define $Q_0(x)=1$.  We have
\bean
\label{Q hy 2}
Q_s(x) &=& -\, {}_2F_1\Big(\frac{2-p^s}3, \frac{1-2p^s}3; 1; x\Big), \qquad\! \!\!\on{odd}\,s,
\\
\notag
Q_s(x) &=&  {}_2F_1\Big(\frac{2-2p^s}3, \frac{1-p^s}3; 1; x\Big),\ 
\qquad \on{even}\,s.
\eean
Here we use the fact that for $p=3\ell+2>2$ we have
$(-1)^{(2p^s-1)/3}=-1 $ for odd $s$ and $(-1)^{(p^s-1)/3}=1$ for even $s$.

\vsk.2>

The polynomial $Q_s(x)$ is a solution of the hypergeometric equation \eqref{HE/3} modulo $p^s$. This 
follows from Theorem  \ref{thm m} or  from formula \eqref{Q hy 2}.

\vsk.2>
The master polynomial for $J^{(D)}(x)$ is given by the formulas
\bea
\Psi_s(t,x) &=& t^{(p^s-2)/3}(t-1)^{(p^s-2)/3}(t-x)^{(2p^s-1)/3}, \qquad\, \on{odd}\,s,
\\
\Psi_s(t,x) &=& t^{2(p^s-1)/3}(t-1)^{2(p^s-1)/3}(t-x)^{(p^s-1)/3}, \qquad \!\on{even}\,s.
\eea
The $p^s$-approximation polynomial $R_s(x)$ is defined as the coefficient of 
$t^{p^s-1}$ in $\Phi_s(t,x)$, 
\bean
\label{roe}
R_s(x) &=& (-1)^{(p^s-2)/3}\sum_{k}  \binom{(2p^s-1)/3}{k} \binom{(p^s-2)/3}{k} x^k,
 \qquad \on{odd}\,s,
\\
\notag
R_s(x) &=& (-1)^{2(p^s-1)/3}\sum_{k}  \binom{2(p^s-1)/3}{k} \binom{(p^s-1)/3}{k} x^k,
 \qquad\!\! \on{even}\,s.
\eean
Define $R_0(x)=1$.  We have
\bean
\label{R hy 2}
R_s(x) &=& -\, {}_2F_1\Big(\frac{2-p^s}3, \frac{1-2p^s}3; 1; x\Big), \qquad \on{odd}\,s,
\\
\notag
R_s(x) &=&  {}_2F_1\Big(\frac{2-2p^s}3, \frac{1-p^s}3; 1; x\Big),\quad \
\qquad \!\!\on{even}\,s.
\eean
Here we use the fact that for $p=3\ell+2>2$ we have
$(-1)^{(p^s-2)/3}=-1$ for odd $s$ and $(-1)^{(2p^s-1)/3}=1 $ for even $s$.

\vsk.2>
The polynomial $R_s(x)$ is a solution of the hypergeometric equation \eqref{HE/3} modulo $p^s$. This 
follows from Theorem  \ref{thm m} or  from formula \eqref{R hy 2}.

\begin{lem}
Master polynomials satisfy baby congruences,
\bean
&&
\label{hs}
\Phi_{s+1}(t,x) \Psi_{s-1}(t^p,x^p) \equiv \Phi_{s}(t,x) \Psi_{s}(t^p,x^p) \pmod{p^s}, 
\\
&&
\label{sh}
\Psi_{s+1}(t,x) \Phi_{s-1}(t^p,x^p) \equiv \Psi_{s}(t,x) \Phi_{s}(t^p,x^p) \pmod{p^s} .
\eean
\end{lem}

\begin{proof} We prove \eqref{hs} for an odd  $s$. The case of an even $s$ and congruence \eqref{sh}
are proved similarly. 
The left-hand side of \eqref{hs} for an odd $s=2k+1$ equals
\bea
t^{(p^{s+1}-1)/3}(t-1)^{(p^{s+1}-1)/3}(t-x)^{2(p^{s+1}-1)/3}
t^{2p(p^{s-1}-1)/3}(t^p-1)^{2(p^{s-1}-1)/3}(t^p-x^p)^{(p^{s-1}-1)/3},
\eea
while the right-hand side equals 
\bea
t^{(2p^s-1)/3}(t-1)^{(2p^s-1)/3}(t-x)^{(p^s-2)/3}
t^{p(p^s-2)/3}(t^p-1)^{(p^s-2)/3}(t^p-x^p)^{(2p^s-1)/3}.
\eea
Now the congruence \eqref{hs} for an odd $s$ follows from  formula \eqref{bin}.
\end{proof}

\begin{thm}
\label{th-QR2}
For $p=3\ell+2>2$ the approximation polynomials $R_s(x)$ and $Q_s(x)$ satisfy the congruences
\bean
\label{QR}
&&
Q_{s+1}(x) R_{s-1}(x^p) \equiv Q_{s}(x) R_{s}(x^p) \pmod{p^s},
\\
\label{RQ}
&&
R_{s+1}(x) Q_{s-1}(x^p) \equiv R_{s}(x) Q_{s}(x^p) \pmod{p^s}.
\eean
\end{thm}

\begin{proof}
We prove \eqref{QR} for an odd  $s$. The case of an even $s$ and congruence \eqref{RQ}
are proved similarly. 
Denote
\bea
f(t,x)
&=&
(t-1)^{(2p-1)/3}(1-x/t)^{(p-2)/3}=(t-1)^{2\ell+1}(1-x/t)^{\ell},
\\
g(t,x)
&=&
(t-1)^{(p-2)/3}(1-x/t)^{(2p-1)/3}=
(t-1)^{\ell}(1-x/t)^{2\ell+1}.
\eea
It is easy to see that for an odd $s$ we have 
\bea
Q_{s+1}(x)  
&=& 
\on{CT}_t\Big[(t-1)^{(p^{s+1}-1)/3}(1-x/t)^{2(p^{s+1}-1)/3}\Big]
\\
&=&
\on{CT}_t\Big[f(t,x) g(t,x)^{p} \dots f(t,x)^{p^{s-1}} g(t,x)^{p^s} \Big],
\\
R_{s-1}(x)  
&=& 
\on{CT}_t\Big[(t-1)^{2(p^{s-1}-1)/3}(1-x/t)^{(p^{s-1}-1)/3}\Big]
\\
&=&
\on{CT}_t\Big[g(t,x) f(t,x)^{p} \dots g(t,x)^{p^{s-3}}f(t,x)^{p^{s-2}} \Big],
\eea
\bea
Q_{s}(x)  
&=& 
\on{CT}_t\Big[(t-1)^{(2p^{s}-1)/3}(1-x/t)^{(p^{s}-2)/3}\Big]
\\
&=&
\on{CT}_t\Big[f(t,x) g(t,x)^{p} \dots g(t,x)^{p^{s-2}} f(t,x)^{p^{s-1}} \Big],
\\
R_{s}(x)  
&=& 
\on{CT}_t\Big[(t-1)^{(p^{s}-2)/3}(1-x/t)^{(2p^{s}-1)/3}\Big]
\\
&=&
\on{CT}_t\Big[g(t,x) f(t,x)^{p} \dots f(t,x)^{p^{s-2}}g(t,x)^{p^{s-1}} \Big],
\eea
Observe that the $(s+1)$-tuple of Laurent polynomials
$(f(t,x),g(t,x),\dots, f(t,x),g(t,x))$
is admissible in the sense of Definition  \ref{defn2}. Now the application of Theorem \ref{thm cg} gives congruence 
\eqref{QR} for an odd $s$.
\end{proof}

\begin{rem}
In general we may take any admissible tuple of Laurent polynomials and obtain
the coresponding Dwork congruences.
For example the tuples $(f,g,f,f,f,g,g,f,\dots)$ and 
$(f,f,\dots)$ are admissible.
\end{rem}

Using formulas \eqref{Q hy 2} and \eqref{R hy 2} we may reformulate congruences \eqref{QR} and \eqref{RQ} as
\bean
\label{odd}
&{}_2F_1\Big(\frac{2-p^{s+1}}3, \frac{1-2p^{s+1}}3; 1; x\Big)
{}_2F_1\Big(\frac{2-p^{s-1}}3, \frac{1-2p^{s-1}}3; 1; x^p\Big) 
\equiv \phantom{aaaaaaaaaaaaaaaaaa}
\\
\notag
&
\equiv
{}_2F_1\Big(\frac{2-2p^s}3, \frac{1-p^s}3; 1; x\Big)
{}_2F_1\Big(\frac{2-2p^s}3, \frac{1-p^s}3; 1; x^p\Big) \pmod{p^s}, \ \  \on{odd} \,s,
\eean
\bean
\label{even}
&{}_2F_1\Big(\frac{2-2p^{s+1}}3, \frac{1-p^{s+1}}3; 1; x\Big)
{}_2F_1\Big(\frac{2-2p^{s-1}}3, \frac{1-p^{s-1}}3; 1; x^p\Big) 
\equiv \phantom{aaaaaaaaaaaaaaaaaa}
\\
\notag
&
\equiv
{}_2F_1\Big(\frac{2-p^s}3, \frac{1-2p^s}3; 1; x\Big)
{}_2F_1\Big(\frac{2-p^s}3, \frac{1-2p^s}3; 1; x^p\Big) \pmod{p^s}, \ \  \on{even} \,s.
\eean
Recall that in these congruences we have $p=3\ell+2$.

\vsk.2>

Consider the $p$-adic presentations
\bean
\label{p-1/3}
-1/3
&=&
2\ell +1 + \ell p+(2\ell+1)p^2+\ell p^3+\cdots ,
\\
\label{p-2/3}
-2/3
&=&
\ell  + (2\ell+1)p+\ell p^2+ (2\ell+1)p^3+\cdots ,
\eean
Recall that $[-1/3]_s$ (resp., $[-2/3]_s$)  is the sum of the first $s$ summands in \eqref{p-1/3} (resp., \eqref{p-2/3}). Then
congruences \eqref{odd} and \eqref{even} imply that for $p=3\ell+2$, \,$s\geq 1$ we have

\bean
\label{13 23} 
&
{}_2F_1\big([-\frac23]_{s+1}, [-\frac13]_{s+1};1;x\big)\,
{}_2F_1\big([-\frac23]_{s-1}, [-\frac13]_{s-1};1;x^p\big)
\equiv \phantom{aaaaaaaaaaaa}
\\
\notag
&
\equiv
{}_2F_1\big([-\frac 23]_{s}, [-\frac13]_{s};1;x\big)\,
{}_2F_1\big([-\frac23]_{s}, [-\frac13]_{s};1;x^p\big)
\pmod{p^s}.
\eean

\subsection{Limits of $\bar Q_s(x)$}

Define
\bean
\label{bar Q}
\bar Q_s(x) = {}_2F_1\Big(\Big[\!-\frac 23\Big]_{s}, \Big[\!-\frac13\Big]_{s};1;x\Big).
\eean
Then for any prime $p>3$ we have
\bean
\label{bar QQ} 
\bar Q_{s+1}(x)\bar Q_{s-1}(x^p)
\equiv
\bar Q_{s}(x)\bar Q_{s}(x^p) \pmod{p^s}
\eean
by \eqref{132+1} and \eqref{13 23}.

\begin{thm}
\label{thm LQ}
For any prime $p>3$ and integer $s\geq 1$ the rational function $\frac{\bar Q_{s+1}(x)}{\bar Q_s(x^p)}$  is regular on the domain
\bean
\label{fD Q}
\frak D  = \{ x\in\Z_p\ |\  |\bar Q_1(x)|_p=1\}.
\eean
The sequence  $\big(\frac{\bar Q_{s+1}(x)}{\bar Q_s(x^p)}\big)_{s\geq 1}$ uniformly converges on $\frak D$.
The limiting analytic function $f(x)$ equals the ratio
$\frac{F(x)}{F(x^p)}$  on the disc $D_{0,1}$ where 
$F(x):={}_2F_1(2/3,1/3;1;x)$ is defined by the corresponding convergent power series.
\end{thm}

\begin{proof}
The proof is the same as the proof of Theorem \ref{thm LP}.
\end{proof}

\subsection{Remark}

Although the congruences \eqref{132+1} and \eqref{13 23} look the same for $p=3\ell +1$ and $p=3\ell +2$ 
the proofs of them are different as presented above.
The proof of \eqref{132+1} for $p=3\ell+1$ 
uses just any one of the two master polynomials: \,   $\Phi_s(t,x)$ or
$\Psi_s(t,x)$, while the proof of \eqref{13 23} for $p=3\ell+2$ uses the interaction of the two master polynomials
$\Phi_s(t,x)$ and $\Psi_s(t,x)$. Cf. the baby congruences 
\eqref{bc 1}, \eqref{hs}, \eqref{sh}.

\subsection{Congruences related to $-\frac15$, $-\frac25$, $-\frac35$, $-\frac45$ }
\label{sec 6}

In this section we formulate the congruences related to the above rational numbers. The proof of
these congruences is similar to
the corresponding proofs in Sections \ref{sec 4} and \ref{sec 5}.

For $p=5\ell \pm 2$ and any $s\geq1$ we have 
\bean
\label{4132} 
&
{}_2F_1\big([-\frac45]_{s+1}, [-\frac15]_{s+1};1;x\big)\,
{}_2F_1\big([-\frac35]_{s-1}, [-\frac25]_{s-1};1;x^p\big)
\equiv \phantom{aaaaaaaaaaaa}
\\
\notag
&
\equiv
{}_2F_1\big([-\frac 45]_{s}, [-\frac15]_{s};1;x\big)\,
{}_2F_1\big([-\frac35]_{s}, [-\frac25]_{s};1;x^p\big)
\pmod{p^s},
\eean
\bean
\label{3241} 
&
{}_2F_1\big([-\frac35]_{s+1}, [-\frac25]_{s+1};1;x\big)\,
{}_2F_1\big([-\frac45]_{s-1}, [-\frac15]_{s-1};1;x^p\big)
\equiv \phantom{aaaaaaaaaaaa}
\\
\notag
&
\equiv
{}_2F_1\big([-\frac 35]_{s}, [-\frac25]_{s};1;x\big)\,
{}_2F_1\big([-\frac45]_{s}, [-\frac15]_{s};1;x^p\big)
\pmod{p^s}.
\eean
For $p=5\ell \pm 1$ and any $s\geq1$ we have 
\bean
\label{4141} 
&
{}_2F_1\big([-\frac45]_{s+1}, [-\frac15]_{s+1};1;x\big)\,
{}_2F_1\big([-\frac45]_{s-1}, [-\frac15]_{s-1};1;x^p\big)
\equiv \phantom{aaaaaaaaaaaa}
\\
\notag
&
\equiv
{}_2F_1\big([-\frac 45]_{s}, [-\frac15]_{s};1;x\big)\,
{}_2F_1\big([-\frac45]_{s}, [-\frac15]_{s};1;x^p\big)
\pmod{p^s},
\eean
\bean
\label{3232} 
&
{}_2F_1\big([-\frac35]_{s+1}, [-\frac25]_{s+1};1;x\big)\,
{}_2F_1\big([-\frac35]_{s-1}, [-\frac25]_{s-1};1;x^p\big)
\equiv \phantom{aaaaaaaaaaaa}
\\
\notag
&
\equiv
{}_2F_1\big([-\frac 35]_{s}, [-\frac25]_{s};1;x\big)\,
{}_2F_1\big([-\frac35]_{s}, [-\frac25]_{s};1;x^p\big)
\pmod{p^s}.
\eean

Similar congruences hold for rational numbers of the form
$a/b$, where $b$ is a prime and $1-b\leq a \leq-1$.
These congruences will be described somewhere else.

\section{KZ equations}
\label{sec 7}

\subsection{KZ equations}
Let $\g$ be a simple Lie algebra with an invariant scalar product.
The {\it Casimir element}  is 
\bea
\Om = {\sum}_i \,h_i\ox h_i \ \ \in \ \g \ox \g,
\eea
where $(h_i)\subset\g$ is an orthonormal basis.
Let  $V=\otimes_{i=1}^n V_i$ be 
a tensor product of $\g$-modules, $\ka\in\C^\times$ a nonzero number.
The {\it  KZ equations} is the system of differential 
equations on a $V$-valued function $I(z_1,\dots,z_n)$,
\bea
\frac{\der I}{\der z_i}\ =\ \frac 1\ka\,{\sum}_{j\ne i}\, \frac{\Om_{i,j}}{z_i-z_j} I, \qquad i=1,\dots,n,
\eea
where $\Om_{i,j}:V\to V$ is the Casimir operator acting in the $i$th and $j$th tensor factors,
see \cite{KZ, EFK}.

\vsk.2>

This system is a system of Fuchsian first order
 linear differential equations. 
  The equations are defined on the complement in $\C^n$ to the union of all diagonal hyperplanes.
 
\vsk.2>

The object of our discussion is the following particular case. We consider
 the following system of differential and algebraic  equations
for a column $3$-vector $I=(I_1,I_2,I_3)$ depending on variables $z=(z_1,z_2,z_3)$\,:
\bean
\label{KZZ}
\frac{\partial I}{\partial z_1}  
&=&
   {\frac 12} \Big(
   \frac{\Omega_{12}}{z_1 - z_2} 
+\frac{\Omega_{13}}{z_1 - z_3} 
      \Big) I ,
\qquad
\frac{\partial I}{\partial z_2}  
= 
   {\frac 12} \Big(
   \frac{\Omega_{21}}{z_2 - z_1} 
+\frac{\Omega_{23}}{z_2 - z_3} 
      \Big) I ,
      \\
\notag
\frac{\partial I}{\partial z_3}  
&=& 
   {\frac 12} \Big(
   \frac{\Omega_{31}}{z_3 - z_1} 
+\frac{\Omega_{32}}{z_3 - z_2} 
\Big) I,
\qquad \quad
0= I_1+I_2+I_3,
\eean 
where $\Om_{ij}=\Om_{ji}$ and
\bea
 \Omega_{12} =  
 \begin{pmatrix}
 -1 & 1 & 0 
 \\
 1   & -1 & 0
 \\
0   & 0 & 0
\end{pmatrix} ,
\quad
 \Omega_{13} =  
 \begin{pmatrix}
 -1 & 0 & 1 
 \\
 0   & 0 & 0
 \\
1   & 0 & -1
\end{pmatrix} ,
\quad
 \Omega_{23} =  
 \begin{pmatrix}
 0 & 0 & 0 
 \\
 0   & -1 & 1
 \\
0   & 1 & -1
\end{pmatrix} .
\eea       
Denote
\bea
&
H_1(z) =   {\frac 12} \Big(
   \frac{\Omega_{12}}{z_1 - z_2} 
+\frac{\Omega_{13}}{z_1 - z_3} 
      \Big),
      \quad
H_2(z)={\frac 12}\Big(
   \frac{\Omega_{21}}{z_2 - z_1} 
+\frac{\Omega_{23}}{z_2 - z_3} 
      \Big),
      \quad
      H_3(z)= {\frac 12}\Big(
   \frac{\Omega_{31}}{z_3 - z_1} 
+\frac{\Omega_{32}}{z_3 - z_2} 
\Big),
\\
&
\nabla_i^{\on{KZ}} = \frac{\der}{\der z_i} - H_i(z), \qquad i=1,2,3.
\eea
Then the KZ equations can be written as the system of equations,
\bea
\nabla_i^{\on{KZ}}I=0, \quad i=1,2,3,\qquad I_1+I_2+I_3 =0.
\eea

\vsk.2>
System  \eqref{KZZ} is the system of  KZ equations 
with parameter $\ka=2$ associated with the Lie algebra $\sll_2$ and the subspace of
 singular vectors of weight $1$ of the tensor power 
$(\C^2)^{\ox 3}$ of two-dimensional irreducible $\sll_2$-modules, 
up to a gauge transformation, see 
this example in  \cite[Section 1.1]{V2}.

\subsection{Solutions over $\C$}
\label{sec 11.4}

Define the {\it master function}
\bea
\Phi(t,z) = (t-z_1)^{-1/2}(t-z_2)^{-1/2}(t-z_3)^{-1/2}
\eea
and the column 3-vector
\bean
\label{KZ sol} 
I^{(C)}(z) = (I_1(z),I_2(z),I_3(z)):=
\int_{C}
\Big(\frac {\Phi(t,z)}{t-z_1}, \frac {\Phi(t,z)}{t-z_2}, \frac {\Phi(t,z)}{t-z_3}\Big)dt
\,,
\eean
where  $C\subset \C-\{z_1,z_2,z_3\}$  
is a contour on which the integrand  takes its initial value when $t$ encircles $C$.

\begin{thm}[{cf.\cite{V4}}]
The function $I^{(C)}(z)$ is a solution of system \eqref{KZZ}.

\end{thm}

This theorem is a very particular case of the results in \cite{SV1}.

\begin{proof}  
The theorem follows from  Stokes' theorem and the two identities:
\bean
\label{i1}
-\frac 12\,
\Big(\frac {\Phi(t,z)}{t-z_1}
 + \frac{\Phi(t,z)}{t-z_2} + \frac {\Phi(t,z)}{t-z_3}\Big)\,  
=\, \frac{\der\Phi}{\der t}(t,z)\,,
\eean
\bean
\label{i2}
\Big(\frac{\der }{\der z_i}-\frac12
\sum_{j\ne i} \frac {\Omega_{i,j}}{z_i-z_j} \Big)
\Big(\frac {\Phi(t,z)}{t-z_1}, \frac{\Phi(t,z)}{t-z_2},\frac {\Phi(t,z)}{t-z_3}\Big)\,  
= \frac{\der \Psi^i}{\der t} (t,z),
\eean
where  $\Psi^i(t,z)$ is the column $3$-vector   $(0,\dots,0,-\frac{\Phi(t,z)}{t-z_i},0,\dots,0)$ with 
the nonzero element at the $i$-th place. 
\end{proof}

\begin{thm} [{cf. \cite[Formula (1.3)]{V1}}]
\label{thm dim}

All solutions of system \eqref{KZZ} have this form. 
Namely, the complex vector space of solutions of the form \eqref{KZ sol} is $2$-dimensional.

\end{thm}

\subsection{Solutions as vectors of first derivatives}
\label{sec 11.5}

Consider the elliptic integral
\bean
\label{Lg}
T(z) = T^{(C)}(z) =
\int_C
\Phi(t,z) dt.
\eean
Then
\bean
\label{KI}
I^{(C)}(z) 
=
\,
2\,
\Big(\frac {\der T^{(C)}}{\der z_1},
\frac {\der T^{(C)}}{\der z_2},
\frac {\der T^{(C)}}{\der z_3}\Big).
\eean
Denote
$\nabla T =
\Big(\frac {\der T}{\der z_1},
\frac {\der T}{\der z_2},
\frac {\der T}{\der z_3}\Big)$.
Then the column gradient vector of the function $T(z)$ satisfies the following system of  (KZ) equations
\bean
\label{kzz}
\nabla_i^{\on{KZ}} \nabla T =0, \quad i=1,2,3,\qquad  
\frac {\der T}{\der z_1} +
\frac {\der T}{\der z_2}+
\frac {\der T}{\der z_3}=0.
\eean
This is a system of second order linear differential equations on the function $T(z)$.

\subsection{Solutions modulo $p^s$}
\label{sec:new}

For an integer $s\geq 1$ define the master polynomial
\bean
\label{KZ mp}
\Phi_s(t,z) = \big((t-z_1)(t-z_1)(t-z_1)\big)^{(p^s-1)/2}.
\eean
Define the column 3-vector
\bean
\label{I_s}
I_s(z)=(I_{s,1}(z), I_{s,2}(z), I_{s,3}(z))
\eean
as the coefficient of $t^{p^s-1}$ in the polynomial
\bean
\label{KZ polyn}
\Big(\frac {\Phi_s(t,z)}{t-z_1}, \frac {\Phi_s(t,z)}{t-z_2}, \frac {\Phi_s(t,z)}{t-z_3}\Big).
\eean

\begin{thm} [\cite{V4}]
\label{thm 7.3}
The polynomial $I_s(z)$ is a solution of system \eqref{KZZ}
modulo $p^s$.

\end{thm}

\begin{proof}

We have the following modifications of identities \eqref{i1}, \eqref{i2}\,:
\bean
\label{M1}
\frac {p^s-1}2\,
\Big(\frac {\Phi_s(t,z)}{t-z_1}
 + \frac{\Phi_s(t,z)}{t-z_2} + \frac {\Phi_s(t,z)}{t-z_3}\Big)\,  
=\, \frac{\der\Phi_s}{\der t}(t,z)\,,
\eean
\bean
\label{M2}
\Big(\frac{\der }{\der z_i} +  \frac {p^s-1}2
\sum_{j\ne i} \frac {\Omega_{i,j}}{z_i-z_j} \Big)
\Big(\frac {\Phi_s(t,z)}{t-z_1}, \frac{\Phi_s(t,z)}{t-z_2},\frac {\Phi_s(t,z)}{t-z_3}\Big)\,  
= \frac{\der \Psi_s^i}{\der t} (t,z),
\eean
where  $\Psi_s^i(t,z)$ is the column $3$-vector   $(0,\dots,0,-\frac{\Phi_s(t,z)}{t-z_i},0,\dots,0)$ with 
the nonzero element at the $i$-th place. Theorem \ref{thm 7.3} follows from these identities.
\end{proof}

\subsection{$p^s$-Approximation polynomials of $T(z)$}
Define the $p^s$-approximation polynomial $T_s(z)$ 
of the elliptic integral $T(z)$ as the coefficient of $t^{p^s-1}$ in the master polynomial
$\Phi_s(t,z)$,
\bean
\label{Ts}
&&
\\
\notag
&&
\!\!\!\!\!
T_s(z) = (-1)^{(p^s-1)/2}\sum_{k_1+k_2+k_3 = (p^s-1)/2} \binom{(p^s-1)/2}{k_1}\binom{(p^s-1)/2}{k_2}\binom{(p^s-1)/2}{k_3}
z_1^{k_1}z_2^{k_2}z_3^{k_3}.
\eean
We put $T_0(x)=1$.

The polynomial $T_s(z_1,z_2,z_3)$ is symmetric with respect to permutations of $z_1,z_2,z_3$
and
\bean
\label{T to P}
T_s(1,z_2,0) = P_s(z_2),
\eean
where $P_s(x)$ is defined in \eqref{Ps}. The gradient vector 
$\nabla T_s :=\Big(\frac{\der T_s}{\der z_1}, \frac{\der T_s}{\der z_2}, \frac{\der T_s}{\der z_3}\Big)$
of the
\linebreak
 $p^s$-approximation 
polynomial $T_s(z)$ is a solution modulo $p^s$ of system \eqref{KZZ} since 
\bean
\label{nqs}
\nabla T_s = \frac{1-p^s}2\,(I_{s,1}(z), I_{s,2}(z), I_{s,3}(z)).
\eean

\begin{lem}
For  $s\geq 1$
the master polynomials satisfy baby congruences,

\bean
\label{T baby}
\Phi_{s+1}(t,z) \Phi_{s-1}(t^p,z_1^p,z_2^p,z_3^p) \equiv \Phi_{s}(t,z) \Phi_{s}(t^p,z_1^p,z_2^p,z_3^p) \pmod{p^s}. 
\eean
\qed
\end{lem}

\begin{thm}
\label{thm T}
For  $s\geq 1$ we have

\bean
\label{TT}
T_{s+1}(z_1,z_2,z_3) \,
T_{s-1}(z_1^p,z_2^p,z_3^p) \,\equiv\, T_{s}(z_1,z_2,z_3)\, T_{s}(z_1^p,z_2^p,z_3^p)
\pmod{p^s}.
\eean

\end{thm}

\begin{proof}
Let
$h(t,z) = t^{1-p} \big((t-z_1)(t-z_2)(t-z_3)\big)^{(p-1)/2}$.
Then
\bea
T_s(z) = \on{CT}_t\big[h(t,z) h(t,z)^p\dots h(t,z)^{p^{s-1}}\big].
\eea
The tuple of Laurent polynomials $(h(t,z),h(t,z), \dots)$
is admissible in the sense of Definition  \ref{defn2}.  
Now the application of Theorem \ref{thm cg} gives congruence 
\eqref{TT}.
\end{proof}

\subsection{Limits of $T_s(z)$}

Denote  $\bar T_s(z) := (-1)^{(p^s-1)/2} T_s(z)$, 
see \eqref{Ts}, 

\bean
\label{fD T}
 \frak D = \{(z_1,z_2,z_3)\in \Z_p^3 \,\mid |\bar T_1(z_1,z_2,z_3)|_p=1\}.
\eean

Notice that
$\Z_p^3 = \bigcup_{\al,\beta,\ga \in\F_p} D_{\al,1}\times D_{\beta,1}\times D_{\ga,1}$.
Since $\bar T_1(z)$ has coefficients in $\Z_p$,
\bea
\frak D ={\bigcup}^o D_{\al,1}\times D_{\beta,1}\times D_{\ga,1}
\eea
where the summation ${\bigcup}^o$ is over all 
$\al,\beta,\ga \in\F_p$ such that  $|T(\om(\al),\om(\beta),\om(\ga))|_p=1$.
For the same reason we have
\bean
\label{Tzp}
 \frak D = \{(z_1,z_2,z_3)\in \Z_p^3 \,\mid |\bar T_1(z_1^p,z_2^p,z_3^p)|_p=1\}.
 \eean
Denote
\bea
 \frak E =\{ (1,z_2,0)\in \Z_p^3\ |\ |z_2|_p<1 \}.
 \eea
 
 \begin{lem}
 \label{lem KZ exp}
 We have $ \frak E\subset  D_{1,1}\times D_{0,1}\times D_{0,1}\subset \frak D$.
  
 \end{lem}

\begin{proof}  The first inclusion is clear. The second inclusion follows from the equality
\\
$\bar T_s(1,0,0)=1$.
\end{proof}

\begin{thm}
\label{thm LT}

For $s\geq 1$ the rational function $\frac{\bar T_{s+1}(z)}{\bar T_s(z^p)}$  is regular on $\frak D$.
The sequence  $\big(\frac{\bar T_{s+1}(z)}{\bar T_s(z^p)}\big)_{s\geq 1}$ uniformly converges on 
$\frak D$.
The limiting analytic function $f(z)$,
restricted to $\frak E$, equals the ratio
$\frac{F(z_2)}{F(z_2^p)}$   where 
$F(z_2):={}_2F_1(1/2,1/2;1;z_2)$ is defined by the convergent power series  \eqref{111}.
We also have $|f(z)|_p=1$ for every $z\in \frak D$.

\end{thm}

\begin{proof}
By Lucas' theorem  $\bar T_s(z) \equiv \bar T_1(z) \bar T_1(z^p)\dots \bar  T_1(x^{p^{s-1}})$ $\pmod{p}$. 
Hence
$|\bar  T_s(z)|_p = |\bar  T_s(z^p)|_p=1$ for $s\geq 1$, $z\in\frak D$. Hence the rational 
functions 
$\frac{\bar T_{s+1}(z)}{\bar T_s(z^p)}$  are regular on $\frak D$

Congruence \eqref{TT} implies that
\bean
\label{TTr}
\Big|\frac{\bar T_{s+1}(z)}{\bar T_{s}(z^p)} - \frac{\bar T_{s}(z)}{\bar T_{s-1}(z^p)}\Big|_p\leq p^{-s}
\quad\on{for}\ z\in\frak D.
\eean
This shows the uniform convergence of $\big(\frac{\bar T_{s+1}(z)}{\bar T_s(z^p)}\big)_{s\geq 1}$
 on  $\frak D$.
For the limiting function 
$f(z)$  we have
$|f(z)|_p=1$ for $z\in \frak D$.

We have
$\bar T_s(1,z_2,0) = \bar P_s(z_2) = \sum_{k} \binom{(p^s-1)/2}{k}^2z_2^k$.
Clearly, for any fixed index $k$ the coefficient $\binom{(p^s-1)/2}{k}^2$ of $z_2^k$ in $\bar T_s(1,z_2,0)$
converges $p$-adically  to the coefficient $\binom{-1/2}{k}^2$ of $z_2^k$ in $F(z_2)$.
Hence  the sequence $(\bar T_s(1,z_2,0))_{s\geq 1}$ converges to $F(z_2)$ on $\frak E$, so that
$f(1,z_2,0) =\frac{F(z_2)}{F(z^p)}$  on  $\frak E$.
The theorem is proved.
\end{proof}

\begin{rem}
The analytic function $f(z)$ of Theorem \ref{thm lim} exhibits behavior very different from the behavior
 of the corresponding ratio $T^{(C)}(z)/T^{(C)}(z^p)$ of  complex elliptic integrals.  
 
By Theorem \ref{thm lim} the function $f(z)$ restricted to the one-dimensional discs 
 $\{ (z_1,0,1)\in \Z_p^3\ |\ |z_1|_p<1 \}$, $\{ (1,z_2,0)\in \Z_p^3\ |\ |z_2|_p<1 \}$,
$\{ (0,1,z_3)\in \Z_p^3\ |\ |z_3|_p<1 \}$ equals $\frac{F(z_1)}{F(z_1^p)}$ , $\frac{F(z_2)}{F(z_2^p)}$ ,
$\frac{F(z_3)}{F(z_3^p)}$, respectively. 

In the complex case, for the ratio $T^{(C_1)}(z)/T^{(C_1)}(z^p)$ 
to be equal to $\frac{F(z_1)}{F(z_1^p)}$  on $\{ (z_1,0,1)\in \C^3\ |\ |z_1|<1 \}$,
 the contour $C_1$ must be the cycle 
 on the elliptic curve $y^2=(t-z_1)t(t-1)$ vanishing at $z_1=0$. Similarly
for $T^{(C_2)}(z)/T^{(C_2)}(z^p)$ to be equal to $\frac{F(z_2)}{F(z_2^p)}$ on $\{ (1,z_2,0)\in \C^3\ |\ |z_2|<1 \}$,
  the contour $C_2$ must be the cycle 
   on the elliptic curve $y^2=(t-1)(t-z_2)t$ vanishing at $z_2=0$,
   and
for $T^{(C_3)}(z)/T^{(C_3)}(z^p)$ to be equal to $\frac{F(z_3)}{F(z_3^p)}$ on $\{ (0,1,z_3)\in \C^3\ |\ |z_3|<1 \}$, 
  the contour $C_3$ must be the cycle 
   on the elliptic curve $y^2=t(t-1)(t-z_3)$ vanishing at $z_3=0$.
   But these three local complex analytic functions are not restrictions of a single univalued
   complex analytic function
   due to the irreducibility of the monodromy representation of the Gauss--Manin connection associated
   with  the      family of elliptic curves $y^2=(t-z_1)(t-z_2)(t-z_3)$.

\end{rem}

For $i,j\in\{1,2,3\}$ and $s\geq 1$ denote
\bea
f_s(z)=T_{s}(z)\Big/T_{s-1}(z^p),
\qquad
\eta^{(i)}_s(z) = \frac{\der T_s}{\der z_i}(z)\Big/T_s(z), 
\qquad
\eta^{(ij)}_s(z) = \frac{\der^2 T_s}{\der z_i\der z_j}(z)\Big/T_s(z).
\eea

\begin{thm}
\label{thm ddT}

For $s\geq 1$ the rational functions 
$\eta^{(i)}_s(z)$ and $\eta^{(ij)}_s(z)$ are regular on $\frak D$.
The sequences of rational functions
$(\eta^{(i)}_s(z))_{s\geq 1}$ and  $(\eta^{(ij)}_s(z))_{\geq 1}$
converge uniformly on $\frak D$ to analytic functions.
If $\eta^{(j)}$ and $\eta^{(ij)}$ denote the corresponding limits, then
\bean
\label{rel eta}
&
\eta^{(1)} + \eta^{(2)} + \eta^{(3)}=0,
\\
\label{rel eta2}
&
\eta^{(j1)} + \eta^{(j2)} + \eta^{(j3)}=0,  \qquad j=1,2,3,
\\
\label{rel eta3}
&
\frac{\der}{\der z_j}\eta^{(i)}=\eta^{(ji)} - \eta^{(i)}\eta^{(j)}.
\eean

\end{thm}

\begin{proof}
Denote
$\delta_i=z_i\frac{\der}{\der z_i}$.\
By Theorem \ref{thm LT} the sequence $(f_s)$ uniformly converges to 
the analytic function $f$
on $\frak D$.
Therefore, the sequence of the derivatives $(\frac{\der}{\der z_i}f_s)$ uniformly
converges on $\frak D$ to $\frac{\der}{\der z_i}f$.
Hence the  sequence $\big((\delta_if_s)/f_s\big)$ uniformly
converges  on $\frak D$ to the function $(\delta_if)/f$.
At the same time
\bea
\frac{\delta_if_s}{f_s}(z)
=\frac{\delta_i T_s}{T_s}(z)- p\, \frac{\delta_i T_{s-1}}{T_{s-1}}(z^p)
\eea
and, more generally,
\bea
\frac{\delta_if_{s-k}}{f_{s-k}}(z^{p^k})
=\frac{\delta_iT_{s-k}}{T_{s-k}}(z^{p^k})
-p\,\frac{\delta_iT_{s-k-1}}{T_{s-k-1}}(z^{p^{k+1}})
\qquad\text{for}\quad k=0,1,\dots,s.
\eea
Summing the relations up with suitable weights to get telescoping we obtain, for any $r\le s$,
\bea
\sum_{k=0}^{r-1}\,p^k\, \frac{\delta_if_{s-k}}{f_{s-k}}(z^{p^k})
\,=\, \frac{\delta_iT_s}{T_s}(z)\,-\,p^r\, \frac{\delta_i T_{s-r}}{T_{s-r}}(z^{p^r}).
\eea
Choosing $r=[s/2]$ and taking the limit as $s\to\infty$ on both sides 
we arrive at
\bea
\sum_{k=0}^\infty \,p^k\,\frac{\delta_i f}{f}(z^{p^k})
=\lim_{s\to\infty} \frac{\delta_iT_s}{T_s}(z).
\eea
The series on the left uniformly converges on $\frak D$.
Hence there exists the limit on the right-hand side. This means that
\bea
\eta^{(i)}(x) =\lim_{s\to\infty} \frac{\frac{\der}{\der z_i}T_s}{T_s}(z)
=\frac1{z_i}\sum_{k=0}^\infty 
\,p^k\,\frac{\delta_i f}{f}(z^{p^k}) .
\eea
One can further differentiate the resulting equality with respect to any 
of the variables $z_1,z_2,z_3$ to get by induction formulas for $\eta^{(ij)}$
and more generally for $\eta^{(ijk\dots)}$.
Note that \eqref{rel eta3} comes out from differentiating logarithmic derivatives.

Formulas \eqref{rel eta} and \eqref{rel eta2} follow from \eqref{M1}. The theorem is proved.
\end{proof}

\begin{thm}
\label{thm lim}
 We have  the following system of equations on $\frak D$\,:
\bean
\label{eq eta}
&\phantom{aaaa}
\begin{pmatrix}
\eta^{(11)} 
 \\
\eta^{(12)}
 \\
\eta^{(13)}
\end{pmatrix}
=   {\frac 12} \Big(
   \frac{\Omega_{12}}{z_1 - z_2} 
+\frac{\Omega_{13}}{z_1 - z_3} 
      \Big)
\begin{pmatrix}
\eta^{(1)} 
 \\
\eta^{(2)}
 \\
\eta^{(3)}
\end{pmatrix},
\qquad
\begin{pmatrix}
\eta^{(21)} 
 \\
\eta^{(22)}
 \\
\eta^{(23)}
\end{pmatrix}
=   {\frac 12} \Big(
   \frac{\Omega_{21}}{z_2 - z_1} 
+\frac{\Omega_{23}}{z_2 - z_3} 
      \Big)
\begin{pmatrix}
\eta^{(1)} 
 \\
\eta^{(2)}
 \\
\eta^{(3)}
\end{pmatrix},
\\
\notag
&
\begin{pmatrix}
\eta^{(31)} 
 \\
\eta^{(32)}
 \\
\eta^{(33)}
\end{pmatrix}
=   {\frac 12} \Big(
   \frac{\Omega_{31}}{z_3 - z_1} 
+\frac{\Omega_{33}}{z_3 - z_2} 
      \Big)
\begin{pmatrix}
\eta^{(1)} 
 \\
\eta^{(2)}
 \\
\eta^{(3)}
\end{pmatrix},
\qquad
\eta^{(1)}+\eta^{(2)}+ \eta^{(3)} =0.
\eean

\end{thm}

\begin{proof}
The theorem follows from Theorems \ref{thm 7.3} and \ref{thm lim}.
\end{proof}

\begin{thm}
\label{thm nonzero}

The column vector  
\bean
\label{vec eta}
\vec \eta(z):=(\eta^{(1)}(z), \eta^{(2)}(z), \eta^{(3)}(z))
\eean
is  nonzero at every point $z\in\frak D$.

\end{thm}

\begin{proof}

On the one hand, if $\vec \eta(a) =0$ for some $a\in\frak D$, then all derivatives of
$\vec \eta(z)$ at $a$ are equal to zero. This follows from the first three equations in
\eqref{eq eta} written as
\bean
\label{KZ eta}
\frac\der{\der z_i} \vec \eta = (H_i -\eta^{(i)})\,\vec\eta,\qquad i=1,2,3.
\eean
Hence $\vec \eta(z)$ equals zero identically on $\frak D$. 
On the other hand, $\eta^{(2)}(1,0,0) = F'(0)/F(0) = 1/4$ by Theorem \ref{thm LT}.
This contradiction implies the theorem. 
\end{proof}

\subsection{Subbundle $\mc L  \,\to\,  \frak D$} 

Denote $W=\{(I_1,I_2,I_3)\in \Q_p^3\ |\ I_1+I_2+I_3=0\}$.  
The differential operators $\nabla^{\on{KZ}}_i$, $i=1,2,3$, define a connection on
the trivial bundle $W\times \frak D \to \frak D$,
called the KZ connection. The KZ connection is flat, 
\bea
\big[\nabla^{\on{KZ}}_i, \nabla^{\on{KZ}}_j\big]=0 \qquad \forall\,i,j.
\eea
The flat sections of the KZ connection are solutions of  system  \eqref{KZZ} of KZ equations.
 
\vsk.2>
For any $a\in \frak D$ let $\mc L_a \subset W$ be the 
one-dimensional vector subspace generated by $\vec \eta(a)$. Then
\bea
\mc L := \bigcup\nolimits_{a\in  \frak D}\,\mc L_a \,\to\,  \frak D
\eea
is an analytic line subbundle of the trivial bundle
$W\times \frak D \to \frak D$.

\begin{thm}
\label{thm inv} 

The subbundle 
$\mc L  \,\to\,  \frak D$   is invariant with respect to the KZ connection.
In other words, if $s(z)$ is any section of $\mc L  \,\to\,  \frak D$, then
sections $\nabla_i s(z)$, $i=1,2,3$, also are sections of $\mc L  \,\to\,  \frak D$.

\end{thm}

\begin{proof} 
The theorem follows from equations \eqref{KZ eta}.
\end{proof}

\begin{rem}
For any  $a\in\frak D$  we may find locally a scalar analytic  function $u(z)$ such that
$u(z)\cdot\vec \eta(z)$ is a solution of the KZ equations \eqref{KZZ}. Such a function is a solution of
the system of equations $\frac {\der u}{\der z_i} = - \eta^{(i)} u$, $i=1,2,3$.
This system is compatible since $\frac {\der \eta^{(j)}}{\der z_i} = \eta^{(ij)} - \eta^{(i)}\eta^{(j)}
=
\frac {\der \eta^{(i)}}{\der z_j}$.

\end{rem}

\begin{rem}
The corresponding complex KZ connection 
does not have invariant line subbundles due to irreducibility of the monodromy of the KZ connection,
which in our case is the Gauss--Manin connection of the family $y^2=(t-z_1)(t-z_2)(t-z_3)$. 
Thus the existence of the KZ invariant line subbundle $\mc L \to \frak D$ is a
pure $p$-adic feature.

\end{rem}

\begin{rem} 
The invariant subbundles of the KZ connection over $\C$ usually are related to some additional conformal block
constructions, see \cite{FSV, SV2, V3}. Apparently the subbundle $\mc L \to \frak D$  is of a
different $p$-adic nature, cf. \cite{V4}.

\end{rem}

\begin{rem}
Following Dwork we may expect that locally at any point $a\in\frak D$, the
 solutions of the KZ equations of the form $u(z)\cdot\vec \eta(z),$ where $u(z)$ is a scalar function, 
are given at $a$ by power series in 
$z_i-a_i$, $i=1,2,3$,  bounded in their polydiscs of convergence, 
while any other local solution at $a$ is given by a power series unbounded in its polydisc of convergence,
cf. \cite{Dw} and \cite[Theorem A.4]{V4}.
\end{rem}

\subsection{Other definitions of subbundle $\mc L \to \frak D$}

\subsubsection{Line subbundle $\mc M \to \frak D$}

Define a polynomial $U_s(z)$ as the coefficient of $t^{p^s-1}$ in the master polynomial
$\Phi_s(t+z_3,z) = \big((t-(z_1-z_3))(t-(z_2-z_3))t\big)^{(p^s-1)/2}$.
We have  $U_1(z)=T_1(z)$ $\pmod{p}$, by Lucas' theorem.  
Similarly to Theorem \ref{thm T} we conclude that

\bean
\label{UU}
\phantom{aaa}
U_{s+1}(z_1,z_2,z_3)\, U_{s-1}(z_1^p,z_2^p,z_3^p) \,\equiv\,
U_{s}(z_1,z_2,z_3)\, U_{s}(z_1^p,z_2^p,z_3^p) \pmod{p^s}.
\eean

${}$

\noindent
Hence the sequence $\big(\frac {U_{s+1}(z)}{U_s(z^p)}\big)_{s\geq 1}$ uniformly converges to an analytic function 
on the domain  $\frak D$ defined in \eqref{fD T}.
The vector-valued polynomial
$\nabla  U_s(z) =\big(\frac{\der U_s}{\der z_1}, \frac{\der U_s}{\der z_2}, 
\frac{\der U_s}{\der z_3} \big)$ 
is a solution modulo $p^s$ of the KZ equations  \eqref{KZZ}; see \cite[Theorem 9.1]{V4}, 
cf. the proof of Theorem \ref{thm 7.3}. Consider the function
\bea
\vec{\mu}=(\mu^{(1)}, \mu^{(2)},\mu^{(3)}):=\lim_{s\to\infty} 
\frac{\nabla U_s}{U_s}
\eea
defined on the same domain $\frak D$. 
Similarly to the proofs of Theorems \ref{thm ddT}--\ref{thm inv} we conclude that the function
$\vec{\mu}(z)$ is nonzero on $\frak D$ and its values span an analytic line subbundle
\bea
\mc M := \bigcup\nolimits_{a\in  \frak D}\,\mc M_a \,\to\,  \frak D
\eea
of the trivial bundle $W\times \frak D\to\frak D$; \ 
here $\mc M_a\subset W$ is the one-dimensional subspace generated by $\vec\mu(a)$.
The line subbundle $\mc M  \,\to\,  \frak D$ is invariant with respect to the KZ connection.

\begin{thm}
\label{thm L1=2} 

The line bundles $\mc M  \,\to\,  \frak D$ and $\mc L  \,\to\,  \frak D$ coincide.

\end{thm}

\begin{proof}
The proof rests on the two lemmas.

\begin{lem}
\label{lem 1=2}
The line bundles  $\mc M  \,\to\,  \frak D$ and $\mc L  \,\to\,  \frak D$ coincide, 
if there is $a\in\frak D$ such that  $\mc M_a=\mc L_a$.
\end{lem}

\begin{proof}
Let  $\mc M_a=\mc L_a$   for some $a\in\frak D$. Then
$\mc M_z = \mc L_s$ in some neighborhood of $a$,
since locally the subbundles are generated by the values of the solutions with the same
initial condition at $z=a$. Hence $\mc M_z = \mc L_z$ on $\frak D$.
\end{proof}

\begin{lem}
\label{lem =} 

For $i=1,2,3$ the functions
$\frac{\der T_s}{\der z_i}(z)\big/ T_s(z)$ and $\frac{\der U_s}{\der z_i}(z)\big/ U_s(z)$
are equal on the line $z_1=1$, $z_3=0$.
\end{lem}

Hence $\mc M=\mc L$ over the points of that line and, therefore, $\mc M_z = \mc L_z$   for $z\in \frak D$.
\end{proof}

\subsubsection{Line subbundle $\mc N\to \hat{\frak D}$}

Let $\om(x) =F'(x)/F(x)$, where $F(x) = {}_2F_1(1/2,1/2;1;x)$.
We have
$\om(x) = \lim_{s\to\infty}\frac{P_s'(x)}{P_s(x)}$ on $D_{0,1}$. 
Introduce new variables
\bean
\label{id co}
u_1=z_1-z_3, \qquad
u_2=\frac{z_2-z_3}{z_1-z_3},
\qquad
u_3=z_1+z_2+z_3,
\eean
and a vector-valued function
\bean
\label{K/K}
\vec\om(u) 
 = \frac1{u_1}
\big(\!- 1/2 - \om(u_2) u_2,\, \om(u_2),\, 1/2 + \om(u_2)(u_2-1)\big)\,.
\eean
Define 
\bea
\hat{\frak D}_0 = \{(z_1,z_2,z_3)\in \Q_p^3\ |\ z_i\ne z_j\,\ \forall i\ne j\}.
\eea
For any $\si=(i,j,k)\in S_3$ define
\bea
&&
\phantom{aaaaa}
\hat {\frak D}_1^\si =\Big\{(z_1,z_2,z_3)\in \hat {\frak D}_0\  \Big|\ \
\frac{z_j-z_k}{z_i-z_k}\in \Z_p,\ \Big| g\Big(\frac{z_j-z_k}{z_i-z_k}\Big)\Big|_p = 1\Big\},
\\
&&
\hat{\frak D}_2^\si = \Big\{(z_1,z_2,z_3)\in \hat {\frak D}_0\  \Big|\ 
\frac{z_i-z_k}{z_j-z_k} \in \hat {\frak D}_1^\si\Big\},
\qquad
\hat{\frak D}^\si= \hat{\frak D}_1^\si\cup\hat {\frak D}_2^\si,
\qquad
\hat{\frak D} = \sum_{\si\in S_3}\hat{\frak D}^\si,
\eea
where   $g(\la)$ is the Igusa polynomial in \eqref{Igu}.

Using Dwork's results in \cite{Dw}, 
it is shown in \cite[Appendix]{V4} that the values of the analytic continuation of the function $\vec\om(u)$
generate a  line bundle $\mc N\to \hat{\frak D}$  invariant with respect to the KZ connection.

\begin{thm}
\label{thm hat=tilde}

The line bundles $\mc M\to \frak D$ and $\mc N \to \hat{\frak D}$ coincide on 
$\frak D\cap \hat{\frak D}$.

\end{thm}

Thus we identified the line bundles
$\mc L\to \frak D$,  $\mc M\to \frak D$, and $\mc N \to \hat{\frak D}$
over $\frak D\cap \hat{\frak D}$.

\begin{proof}
We have
\bea
&
U_s(z) = (z_1-z_3)^{(p^s-1)/2} P_s\big(\frac{z_2-z_3}{z_1-z_3}\big),
\qquad
\frac{\der U_s}{\der z_2} = (z_1-z_3)^{(p^s-1)/2-1} P_s'\big(\frac{z_2-z_3}{z_1-z_3}\big)
\\
&
\frac{\der U_s}{\der z_1} = \frac{p^s-1}2\,\on{const}\,(z_1-z_3)^{(p^s-1)/2-1} P_s\big(\frac{z_2-z_3}{z_1-z_3}\big)-
\on{const}\,(z_1-z_3)^{(p^s-1)/2-1} \frac{z_2-z_3}{z_1-z_3} P_s'\big(\frac{z_2-z_3}{z_1-z_3}\big).
\eea
Hence
\bea
&
\frac 1{U_s(z)}\Big(\frac{\der U_s}{\der z_1}, \frac{\der U_s}{\der z_2}, 
\frac{\der U_s}{\der z_3}\Big) =
\frac 1{u_1}\Big(\frac{p^s-1}2 - u_2 \frac{P_s'(u_2)}{P_s(u_2)},\,\frac{P_s'(u_2)}{P_s(u_2)},
\,-\frac{p^s-1}2 + (u_2-1) \frac{P_s'(u_2)}{P_s(u_2)}\Big) .
\eea
Clearly  the limit of this vector equals $\vec\om(u)$ as $s\to\infty$. The theorem is proved.
\end{proof}

\section{Concluding remarks}
\label{sec 8}

\subsection{Conjectural stronger congruences for $\bar P_s(x)$}
\label{sec 8.1}

By Theorem \ref{conj B} we have for polynomials $\bar P_s(x) := (-1)^{(p^s-1)/2} P_s(x)$:
\bea
\bar P_{4}(x) \bar P_{2}(x^p) - \bar P_{3}(x) \bar P_{3}(x^p) \equiv 0\pmod{p^3}.
\eea
In particular, for the coefficient of $x^{N_0+N_1p+N_2p^2 +N_3p^3}$
in $\bar P_4(x) \bar P_2(x^p) -\bar P_3(x)\bar P_3(x^p)$ we have
\bean
\label{42=33}
&&
\sum_{\satop{k_1+l_1=N_1}{k_2+l_2=N_2}}
\bigg(\binom{(p^4-1)/2}{N_0+k_1p+k_2p^2+N_3p^3}^2
\binom{(p^2-1)/2}{l_1+l_2p}^2 -
\\
\notag
&&
\phantom{aaaaaaaa}
-
\binom{(p^3-1)/2}{N_0+k_1p+k_2p^2}^2
\binom{(p^3-1)/2}{l_1+l_2p+N_3p^2}^2\,\bigg) \equiv 0 \pmod{p^3}.
\eean
Computer experiments show that this sum can be split into subsums with at most
four terms so that each a subsum is divisible by $p^3$.
More precisely, let $0\leq a, b, \,c,c',d,d'\leq p-1$  be integers.
Define
\bea
&
A(a,b;c,c';d,d') =   \binom{(p^4-1)/2}{a+cp+dp^2+bp^3}^2
\binom{(p^2-1)/2}{c'+d'p}^2 
-\binom{(p^3-1)/2}{a+cp+dp^2}^2
\binom{(p^3-1)/2}{c'+d'p+bp^2}^2
\eea
and
\bean
\label{1}
&&
B(a,b;c,c';d,d') =  
 \Sym A(a,b;c,c';d,d'):=
\\
\notag
&&
\phantom{aaa}
:=A(a,b;c,c';d,d') + A(a,b;c',c;d,d')+A(a,b;c,c';d',d)+A(a,b;c',c;d',d).
\eean
We expect that
the integer $B(a,b;c,c';d,d')$ is divisible by $p^3$.
\vsk.2>

More generally, define
\bea
k=(k^{(1)}, \dots,k^{(s)}),\qquad k^{(i)}=(k^{(i)}_1, k^{(i)}_2),
\eea
\bea
&
A(a,b;k) = \binom{(p^{s+2}-1)/2}{a+\sum_{i=1}^sk^{(i)}_1p^i+bp^{s+1}}^2
\binom{(p^{s}-1)/2}{\sum_{i=1}^sk^{(i)}_2p^{i-1}}^2 
-\binom{(p^{s+1}-1)/2}{a+\sum_{i=1}^sk^{(i)}_1p^i}^2
\binom{(p^{s+1}-1)/2}{\sum_{i=1}^sk^{(i)}_2p^{i-1}+bp^{s}}^2.
\eea
Set
\bea
B(a,b;k) = \Sym A(a,b;k),
\eea
where $\Sym$ denotes the symmetrization with respect to the index $j$
in $k^{(i)}_j$ in each group $k^{(i)}=(k^{(i)}_1,k^{(i)}_2)$. Thus the symmetrization has $2^s$ summands; the case $s=2$ of this symmetrization is displayed in \eqref{1}.

\begin{conj}
\label{conj strong}
The integer  $B(a,b;k)$ is divisible by $p^{s+1}$.

\end{conj}

This conjecture is supported by computer experiments and is checked for $s=1$ using \cite{Gr}.

\subsection{Dwork crystals}

After the first draft of this preprint was posted, Beukers and Vlasenko kindly pointed out that
some of the congruences in Sections~\ref{sec 4} and \ref{sec 5} can be deduced from the general principles in their papers~\cite{BV} and \cite{Vl}.
For example, the coefficient $A_s(t)$ of $(xy)^{p^s-1}$ in the polynomial $(y^2-x(x-1)(x-t))^{p^s-1}$ is equal to
\bea
(-1)^{(p^s-1)/2}\binom{p^s-1}{\frac{p^s-1}2}P_s(x),
\eea
where the polynomial $P_s(x)$ is defined in \eqref{Ps}, hence the congruence
\bea
A_{s+1}(t)A_{s-1}(t^p) \equiv A_s(t)A_s(t^p)  \pmod{p^s}
\eea
which follows from \cite[Theorem~I~(ii)]{Vl} is equivalent to the one in Theorem~\ref{conj B}.
A similar analysis of the coefficient of $(xy^2)^{p^s-1}$ in $(y^3-x(x-1)(x-t)^2)^{p^s-1}$ allows one to produce a different proof of Theorem~\ref{th-QR1} using the main results in~\cite{BV}.
Notice that these arguments require introducing an extra variable, $y$, into the story and make it difficult (if possible)
to establish results like Theorems \ref{thm C_k}, \ref{th-QR2} and congruences \eqref{4132}--\eqref{3232}.
However, there is more than these applications in the methodology of \cite{BV}:
one notable example is the connection of the unit root, which is read off as the $p$-adic limit, with the local zeta function\,---\,see \cite[Part~I, Appendix~A]{BV} for details.

\end{document}